\documentclass[a4paper,10pt]{article}
\usepackage{amssymb,amsmath,amsthm}

\usepackage[dvips]{pstricks}
\usepackage{graphicx}

\usepackage{fullpage}
\font\script=rsfs10 at 11pt
\def\H{{\mbox{\script H}\;}}
\def\div{\mathrm{div}\,}
\def\step#1#2{\par\noindent{\underline{\it Step~#1.}}\emph{ #2}\\}
\def\C{{\mathcal C}}
\def\eps{\varepsilon}
\def\L{{\mathcal L}}
\def\N{{\mathbb N}}
\def\R{{\mathbb R}}
\def\bal{\begin{aligned}}
\def\eal{\end{aligned}}
\newtheorem{Lemma}{Lemma}
\newtheorem{Theorem}[Lemma]{Theorem}

\newtheorem{definition}[Lemma]{Definition}
\newtheorem{Remark}[Lemma]{Remark}

\begin{document}
%
%-------%
% TITLE %
%-------%
%------------------------------------------%
%------------------------------------------%
\title{\textbf{\Large The Cheeger constant of curved strips}}
\author{David Krej\v{c}i\v{r}{\'\i}k$\,^{a, b}$
\ and \ Aldo Pratelli$\,^c$}

\date{\small
%
%\begin{center}
%\begin{quote}
%\begin{enumerate}
%\item[$a)$]
\emph{$a)$
Department of Theoretical Physics,
Nuclear Physics Institute ASCR, \\
25068 \v Re\v z, Czech Republic;
krejcirik@ujf.cas.cz}
\medskip \\
%\item[$b)$]
\emph{$b)$
IKERBASQUE, Basque Foundation for Science, 48011 Bilbao, Spain}
\medskip \\
%\item[$c)$]
\emph{$c)$
Department of Mathematics ``F.~Casorati'',
University of Pavia, \\
Via Ferrata 1, 27100 Pavia, Italy;
pratelli@unipv.it}
%\end{enumerate}
%\end{quote}
%\end{center}
%
\bigskip \\
29 February 2012
}

\maketitle
\begin{abstract}
\noindent
We study the Cheeger constant and Cheeger set for
domains obtained as strip-like neighbourhoods of curves in the plane.
If the reference curve is complete and finite (a ``curved annulus''),
then the strip itself is a Cheeger set
and the Cheeger constant equals the inverse of the half-width of the strip.
The latter holds true for unbounded strips as well,
but there is no Cheeger set.
Finally, for strips about non-complete finite curves,
we derive lower and upper bounds to the Cheeger set,
which become sharp for infinite curves.
The paper is concluded by numerical results for circular sectors.
\end{abstract}
%---------------------%

\section{Introduction}
Let~$\Omega$ be an open connected set in the plane $\R^2$.
The \emph{Cheeger constant} of $\Omega$ is defined as
\begin{equation}\label{Ch.constant}
h(\Omega) := \inf_{S\subseteq\Omega} \frac{P(S)}{|S|} \,,
\end{equation}
where the infimum is taken over all sets~$S\subseteq \Omega$ of finite perimeter. 
Here and in the following, $P(S)$ and~$|S|$ denote the perimeter 
and the area of~$S$, respectively. 
Any minimizer of~\eqref{Ch.constant}, if it exists, 
is called a \emph{Cheeger set} of~$\Omega$ and denoted by~$\C_\Omega$.\par

The problems of existence, uniqueness and regularity of Cheeger sets
have been widely studied in last years, for instance
one may look at~\cite{Kawohl-Fridman_2003,Hebey-Saintier_2006,Saintier,CCN}. We briefly list and discuss here some of the general known properties.

\begin{Theorem}[General known facts]\label{genprop}~
\begin{itemize}
\item[\emph{(i)}]
While for a general $\Omega$ neither existence nor uniqueness is guaranteed, there is always some Cheeger set if $\Omega$ is a bounded open set.
\item[\emph{(ii)}]
If $\Omega_1\subseteq\Omega_2$, then $h(\Omega_1)\geq h(\Omega_2)$, but the strict inclusion does not imply the strict inequality.
\item[\emph{(iii)}]
The boundary of any Cheeger set $\C_\Omega$ intersects the boundary of the set $\Omega$.
\item[\emph{(iv)}]
The part of $\partial\C_\Omega$ which is inside $\Omega$ is made by arcs of circle, all of radius $1/h(\Omega)$, and each of which starts and ends touching the boundary of $\Omega$.
\item[\emph{(v)}] A Cheeger set cannot have corners (i.e., discontinuities in the tangent vector to the boundary giving rise to an angle smaller than $\pi$). In particular, the arcs of circle of $\partial\C\cap \Omega$ must intersect the boundary of $\Omega$ tangentially or in ``open corners'' (i.e., angles bigger than $\pi$).
\item[\emph{(vi)}]
If there is some Cheeger set, there is a connected Cheeger set.
\end{itemize}
\end{Theorem}
Concerning property~(i), examples of non-existence or non-uniqueness can be found for instance in~\cite{Kawohl-Lachand-Robert_2006}, while the existence is immediate by the compactness results for BV functions (see for instance~\cite{Evans-Gariepy,AFP}). Property~(ii) is immediate by the definition~(\ref{Ch.constant}), and examples for the non strict inequality can be found again in~\cite{Kawohl-Lachand-Robert_2006}. Property~(iii) comes immediately by a rescaling of $\C$ with a factor bigger than $1$, since this lowers the ratio in~(\ref{Ch.constant}). Property~(iv) comes from a standard variational argument (see for instance~\cite[Remark~9]{Kawohl-Fridman_2003}). Property~(v) comes directly by noticing that ``cutting a corner'' of a small length $\eps$ decreases $|\C_\Omega|$ by at most $C\eps^2$ and the perimeter by at least $c\eps$. Here, and in the sequel, by ``corner'' we mean a point of the boundary where the tangent vector is discontinuous and makes an angle smaller than $\pi$ (with respect to the internal part of $\Omega$, of course). In the case of angles bigger than $\pi$, we talk about ``open corners'', and they cannot be excluded from $\partial\C$, since for instance, as pointed out in~\cite{Kawohl-Lachand-Robert_2006}, there are open corners
(or ``reentrant corners'' in their terminology) in an $\mathrm{L}$-shaped set.
Finally, property~(vi)
is immediate because if a Cheeger set has different connected components,
each of them must be also a Cheeger set
thanks to the characterization~(\ref{Ch.constant}).\par\bigskip

Apart from the above-mentioned general properties, it is usually a difficult task to find the Cheeger constant or the Cheeger set of a given domain $\Omega$. The situation is simplified when $\Omega$ is a bounded convex set, which is a well-studied particular situation. In fact, in this case it is known that there is a unique open Cheeger set, which is again convex (see~\cite{Alter-Caselles-Chambolle_2005a,Kawohl-Lachand-Robert_2006,CCN}). Moreover, it is also possible to give the following characterization.
\begin{Theorem}[\cite{Kawohl-Lachand-Robert_2006}]\label{Thm.convex}
Let~$\Omega$ be a bounded convex subset of~$\R^2$.
For $r \geq 0$, define
\[
\Omega^r := \{x \in \Omega  \ | \ \mathrm{dist}(x,\partial\Omega) > r \}\,.
\]
There exists a unique value $r=r^*>0$ such that
\begin{equation}\label{distances}
|\Omega^{r}|=\pi r^2 \,.
\end{equation}
Then $h(\Omega)=1/r^*$ and the Cheeger set of~$\Omega$ is the Minkowski sum $\C_\Omega = \Omega^{r^*} + B_{r^*}$, with~$B_{r^*}$ denoting the disc of radius~$r^*$.
\end{Theorem}
\noindent
This theorem can be used to find explicitly $h(\Omega)$ and $\C_\Omega$ in some cases, for example for discs, rectangles and triangles -- in particular, the Cheeger sets of rectangles and triangles are obtained by suitably ``cutting the corners''.
Furthermore, it provides a constructive algorithm for the determination
of the Cheeger constant and Cheeger set for general convex domains,
in particular for convex polygons.

Unfortunately, there is no such a constructive method for non-convex domains. Only one particular case seems to be explicitly known in the literature, namely the annulus, for which it is known that $\C_\Omega=\Omega$. In general, while a trivial strategy to find upper estimates for $h(\Omega)$ is to choose a suitable ``test domain'' $S$ in~(\ref{Ch.constant}), it is less clear how to obtain lower estimates. One possibility is given by the following result concerning ``test vector fields''.
\begin{Theorem}[\cite{Grieser_2006}]\label{Thm.Grieser}
Let $V:\Omega\to\R^2$ be a smooth vector field on~$\Omega$, $h\in\R$, and assume that the pointwise inequalities $|V| \leq 1$ and $\div V \geq h$ hold in~$\Omega$. Then $h(\Omega) \geq h$.
\end{Theorem}
\noindent
An example of applicability of this criterion is the above-mentioned result for the annulus, which can be obtained by employing the vector field of~\cite[Sec.~11, Ex.~4]{Bellettini-Caselles-Novaga_2002} (see also Remark~\ref{Rem.test} below, where the corresponding vector field can be found explicitly). However, for a general set $\Omega$ it is not easy at all to find a vector field producing non-trivial lower bounds by this criterion.\par\bigskip

The purpose of the present paper is to introduce a class of non-convex planar domains for which the Cheeger constant and the Cheeger set can be determined explicitly, namely, the curved strips. This class of sets has been intensively studied in the last two decades as an effective configuration space for curved quantum waveguides (see the survey papers~\cite{DE,KK} and the references therein).\par

More precisely, we call ``curved strip'' a tubular neighbourhood of a curve
without boundary in the plane. There are then few possibilities:
a ``curved annulus'', a ``finite curved strip'', an ``infinite curved strip''
or a ``semi-infinite curved strip''
-- see Figure~\ref{figstrip}
(we leave the formal definitions to Section~\ref{GP}).
\begin{figure}[htbp]
\begin{center}
\begin{tabular}{ccc}
\includegraphics[width=0.3\textwidth]{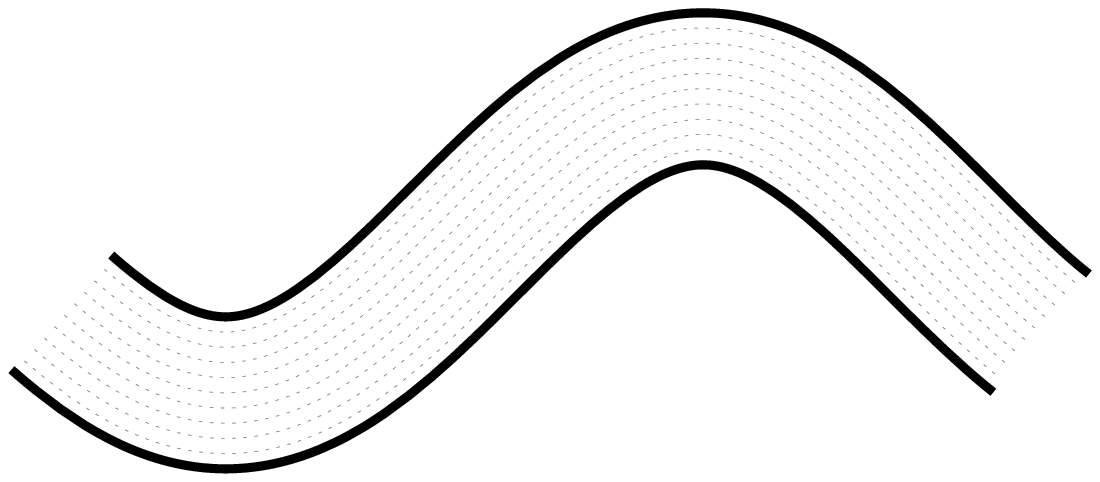}
&& \includegraphics[width=0.3\textwidth]{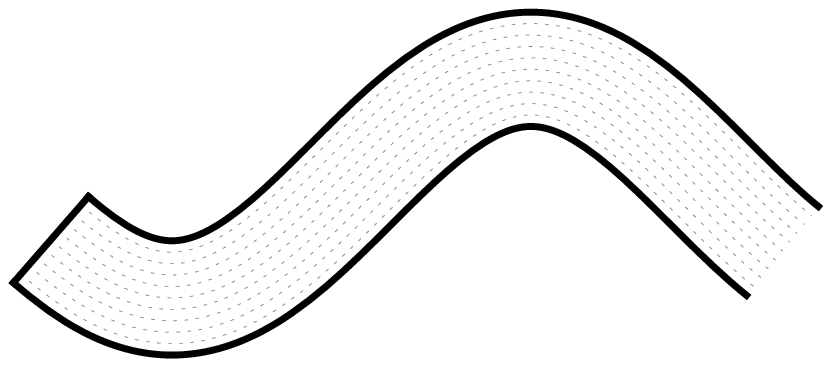}
\\
infinite curved strip
&& semi-infinite curved strip
\\
\includegraphics[width=0.3\textwidth]{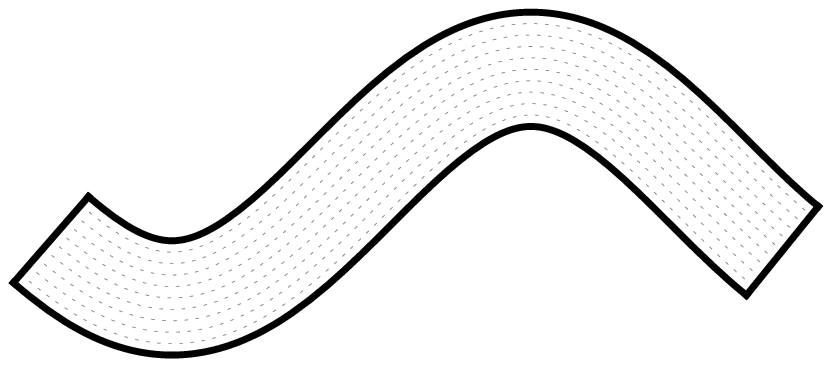}
&& \includegraphics[width=0.3\textwidth]{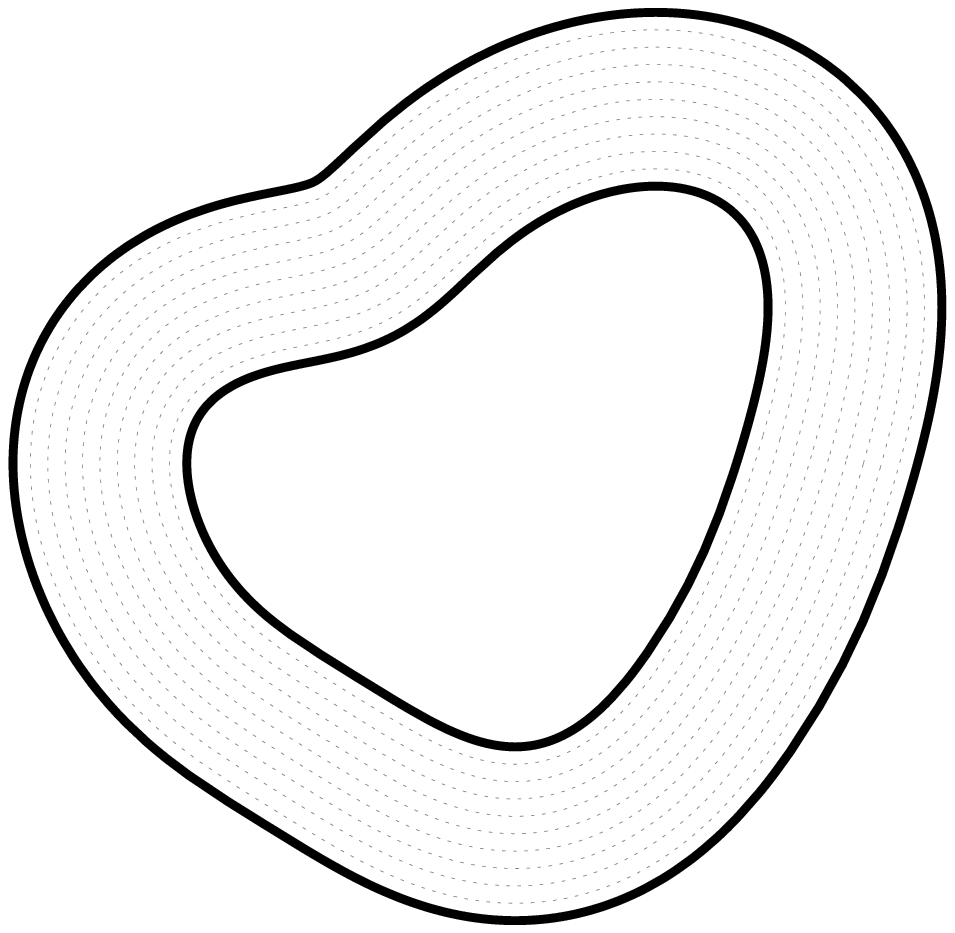}
\\
finite curved strip
&& curved annulus
\end{tabular}
\end{center}
\caption{The four possible types of strips.}
\label{figstrip}
\end{figure}
Our main results, Theorems~\ref{Thm.Cheeger} and~\ref{Thm.bounded},
describe the situation in all these cases.
In particular, for a curved annulus the situation is analogous
to the standard annulus, that is, the strip itself is the unique Cheeger set
and the Cheeger constant only depends on the width of the strip,
irrespectively of the curvature of the curve
-- more precisely, the Cheeger constant is the inverse of the half-width (Theorem~\ref{Thm.Cheeger}, part~(i)). For an infinite or a semi-infinite curved strip, again the Cheeger constant equals the inverse of the half-width of the strip, but there is no Cheeger set (Theorem~\ref{Thm.Cheeger}, part~(ii)). Finally, for a finite curved strip, the situation is analogous to the standard rectangle, that is, there exists a Cheeger set, which is not the whole strip because of the corners, and the Cheeger constant is strictly bigger than the inverse of the half-width. Moreover, in this last case we can also give a (sharp) upper and a lower bound, which only depend on the width and on the length of the strip (Theorem~\ref{Thm.bounded}).\par\bigskip

We conclude this introductory section with a couple of comments. First of all, it is to be mentioned that, in the study of the Cheeger problem, an important role is played by those sets $\Omega$ which are Cheeger sets of themselves. This is what happens in many situations, such as the discs and the annuli and, as we show in the present paper, the ``curved annuli''. Those sets are called \emph{calibrable} and are intensively studied in the image processing literature, see for instance~\cite{Bellettini-Caselles-Novaga_2002}.\par

A second remark has to be done on the connection between the Cheeger constant and the eigenvalue problems. In fact, the \emph{Cheeger inequality} tells that
\begin{equation}\label{Ch.bound}
\lambda_p(\Omega) \geq \left(\frac{h(\Omega)}{p}\right)^p
\end{equation}
for any $p\in(1,\infty)$,
where $\lambda_p(\Omega)$ is the first eigenvalue of the $p$-Laplacian.
Moreover, as shown in~\cite{Kawohl-Fridman_2003},
$h(\Omega)=\lim_{p\searrow 1} \lambda_p(\Omega)$.
At this regards, it is interesting to notice one property of the curved strips.
In fact, it is well known that the first eigenvalue of the Dirichlet Laplacian
(or, more in general, the infimum of the Rayleigh quotient,
in the case of unbounded strips for which there might be no eigenvalues)
for a curved strip strongly depends on its curvature
(see for instance~\cite{DE,EFK,KK}). On the other hand, the Cheeger constant is much less sensitive, since we will show, for instance, that for infinite and semi-infinite curve strips, as well as for curved annuli, the Cheeger constant does not depend at all on the curvature of the strip, but only on its width.

\subsection{The geometrical setting\label{GP}}
In this section we set the notations for the geometrical situation
that we will consider throughout the paper.
Let $\Gamma$ be a ${\rm C}^2$, connected curve in $\R^2$
(i.e., the homeomorphic image of $(0,1)$ or of $\mathbb S^1$ under a ${\rm C^2}$ function),
and let us denote by $|\Gamma| = \int_\Gamma dq$ its length, $dq$ being the arclength element of $\Gamma$. Let also $N:\Gamma\to\R^2$ be a ${\rm C}^1$ vector field giving the normal vector in the points of $\Gamma$, and let $\kappa:\Gamma\to\R$ be the associated curvature (notice that the sign of $\kappa$
depends on the choice of the orientation of $N$).
We recall that to define $\kappa$ it is enough to take a unit-speed parametrization $\gamma$ of $\Gamma$, and hence it is
\begin{equation}\label{frenet}
\kappa(q) = \ddot\gamma\big(\gamma^{-1}(q)\big)\cdot N(q)\,,
\end{equation}
where the dot denotes the standard scalar product in $\R^2$. Now, we introduce a mapping~$\L$ from~$\Gamma\times \R$ to $\R^2$ by
\[
\L(q,t) := q + t \, N(q)\,,
\]
and for any positive~$a$ we introduce the set
\[
\Omega_{\Gamma,a} := \L\big(\Gamma\times(-a,a)\big) \,.
\]
We are interested in the sets $\Omega_{\Gamma,a}$ which are non-self-intersecting tubular neighbourhoods of $\Gamma$.
More precisely, we will always make the assumption that
\begin{equation}\label{Ass.basic}
\hbox{$\L$ is injective in $\Gamma\times [-a,a]$}\,,
\end{equation}
hence the set is as in Figure~\ref{figgeomsetting}. Using the expression for the bilinear form
\begin{equation}\label{metric}
d\L^2 = \big(1-\kappa(q)\,t\big)^2 \, dq^2 + dt^2
\end{equation}
that follows from~(\ref{frenet}),
by the Inverse Function Theorem
we can easily notice that the assumption~(\ref{Ass.basic})
forces $a$ to be small compared to the curvature.
More precisely, (\ref{Ass.basic}) implies that $\big|\kappa(q)\big|\,a \leq 1$
for any $q\in\Gamma$, that the boundary of $\Omega_{\Gamma,a}$ is ${\rm C}^{1,1}$, and that $\L$ is in fact a ${\rm C}^1$ diffeomorphism between $\Gamma\times (-a,a)$ and $\Omega_{\Gamma,a}$.

\begin{figure}[h]
\begin{center}
\includegraphics[width=0.5\textwidth]{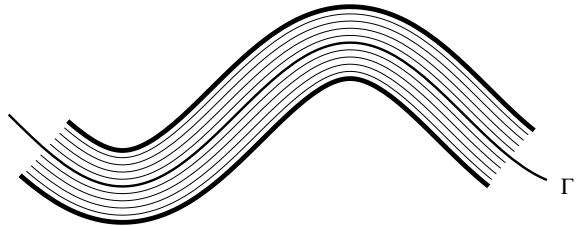}
\end{center}
\caption{The geometry of a curved strip $\Omega_{\Gamma,a}$
and the corresponding curve~$\Gamma$;
the parallel lines correspond to the curves
$s \mapsto \L(s,t)$ with fixed $t\in(-a,a)$.}
\label{figgeomsetting}
\end{figure}

Summing up, under the hypothesis~\eqref{Ass.basic} $\Omega_{\Gamma,a}$ has the geometrical meaning of an open non-self-intersecting strip, contained between the parallel curves $q \mapsto q \pm a \, N(q)$, with $q\in\Gamma$, and it can be identified with the Riemannian manifold $\Gamma\times(-a,a)$ equipped with the metric~\eqref{metric}.\par\medskip

In this paper, we will call \emph{curved strip} any set $\Omega_{\Gamma,a}$
satisfying the assumption~(\ref{Ass.basic}).
Notice that when $\Gamma$ is contained in a line then $\Omega$ reduces to a rectangle,
but the most interesting situation is when $\Gamma$
has a more complicated geometry, since then the associated set is not convex,
hence not covered by the preceding known results for the Cheeger problem.
It is easy to characterize the four possible situations occurring for a curved strip,
to each of which we will associate a name to fix the ideas.
The four kinds of strips are shown in Figure~\ref{figstrip}.
First of all, if the curve $\Gamma$ is not finite,
it may be either infinite or semi-infinite
(that is, not finite but complete,
or not finite and not complete, respectively).
We will call \emph{infinite curved strip} and \emph{semi-infinite curved strip}
the corresponding sets $\Omega_{\Gamma,a}$.
On the other hand, if the curve is finite,
then it can be either compact or not compact
(then homeomorphic to a circle or to an open segment, respectively).
In the first case, we will speak about a \emph{curved annulus},
the annulus corresponding to the case when $\Gamma$ is exactly a circle,
and in the other case about a \emph{finite curved strip}.

\section{The main geometrical results}

In this Section we will give some general technical properties,
which will be used later to show our main results.
First of all, we can easily obtain an upper bound for the curved strips.
In the next result, for a curve $\Gamma$ which is not finite we consider
a unit-speed parametrization
$\gamma:(0,+\infty)\to\R^2$ (respectively, $\gamma:(-\infty,+\infty)\to\R^2$)
if the strip is semi-infinite (respectively, infinite).
Moreover, we will denote by $\Gamma_L$ the subset of $\Gamma$
given by $\gamma(0,L)$ or $\gamma(-L,L)$ for the semi-infinite or infinite case,
respectively.
\begin{Lemma}[Upper bound]\label{Lem.upper}
Let $\Gamma$ be infinite or compact (\emph{i.e.}, $\Omega_{\Gamma,a}$ is a semi-infinite or infinite curved strip, or a curved annulus). Then
\[
h(\Omega_{\Gamma,a}) \leq \frac{1}{a}\,.
\]
In particular, if $\Omega_{\Gamma,a}$ is a curved annulus, then
\[
\frac{P(\Omega_{\Gamma,a})}{|\Omega_{\Gamma,a}|} = \frac{1}{a}\,,
\]
while if $\Omega_{\Gamma,a}$ is a semi-infinite or infinite curved strip, then
\[
\frac{P\big(\Omega_{\Gamma_L,a}\big)}{\big|\Omega_{\Gamma_L,a}\big|} \xrightarrow[L\to\infty]{} \frac{1}{a}\,.
\]
\end{Lemma}
\begin{proof}
If $\Omega_{\Gamma,a}$ is a curved annulus, then we take the whole $S=\Omega_{\Gamma,a}$ as a test domain in~\eqref{Ch.constant}. Recalling~(\ref{metric}), we have then
\[
\frac{P(S)}{|S|}
= \frac{\int_\Gamma (1+\kappa(q)\,a) \, dq + \int_\Gamma (1-\kappa(q)\,a) \, dq} {\int_\Gamma \int_{-a}^a (1-\kappa(q)\,t) \, dt \, dq}
= \frac{2|\Gamma|}{2 a |\Gamma|}
= \frac{1}{a}\,.
\]
Notice that, by the symmetry of the set $S$, the curvature term cancels both in the numerator and in the denominator.\par
On the other hand, if $\Gamma$ is not finite, then the whole strip is not admissible because it has both infinite area and perimeter. However, for any $L>0$, we can consider the finite curved strip $S=\Omega_{\Gamma_L,a}$, which is of course contained in $\Omega_{\Gamma,a}$. Therefore, one can easily evaluate
\begin{equation}\label{approximation}
\frac{P(S)}{|S|}
= \frac{4a + \int_{\Gamma_L} (1+\kappa(q)\,a) \, dq+ \int_{\Gamma_L} (1-\kappa(q)\,a) \, dq}{\int_{\Gamma_L} \int_{-a}^a (1-\kappa(q)\,t) \, dt \, dq}
= \frac{4a + 2|\Gamma_L|}{2 a |\Gamma_L|}
\xrightarrow[L\to\infty]{} \frac{1}{a}\,.
\end{equation}
In the formula for the perimeter, notice the term $4a$ corresponding to the two ``vertical'' parts of $\partial S$ at the start and at the end. Thanks to the definition~(\ref{Ch.constant}), the two above estimates give the thesis.
\end{proof}

The lower bound is much more complicated to obtain. To find it, we will introduce an operation which, in a sense, fills in the ``holes'' and the ``bays'' in the test domains $S$. More precisely, let us take an open set $S\subseteq \Omega_{\Gamma,a}$, and define the set $\Gamma_S$ as
\[
\Gamma_S := \Big\{ q \in \Gamma:\, \L\Big(\{q\}\times (-a,a)\Big)\cap S \neq \emptyset \Big\}
\]
and the functions $f_\pm : \Gamma_S\to [-a,a]$ as
\begin{align*}
f_-(q) := \inf \Big\{ t\in (-a,a):\, (q,t)\in S\Big\}\,, &&
f_+(q) := \sup \Big\{ t\in (-a,a):\, (q,t)\in S\Big\}\,.
\end{align*}
Therefore, $S$ is contained between the two graphs of $f_+$ and $f_-$. Notice now that, if $S$ is connected, then of course so is $\Gamma_S$. In particular, there are two possibilities: either $\Gamma_S$ is a subinterval of $\Gamma$, and in this case we call $q_l$ and $q_r$ its extremes, or $\Gamma_S$ is a closed curve. Observe that if $\Gamma$ is not compact (that is, always except when $\Omega_{\Gamma,a}$ is a curved annulus), then $\Gamma_S$ must necessarily be a subinterval of $\Gamma$; on the other hand, if $\Omega_{\Gamma,a}$ is a curved annulus, then both the situations --that $\Gamma_S$ is a subinterval of $\Gamma$, and that $\Gamma_S$ is a closed curve-- are possible, and in particular $\Gamma_S$ is a closed curve if and only if $\Gamma_S=\Gamma$.

Finally, we can give the following definition.
\begin{definition}\label{defS*}
Let $S$ be an open subset of $\Omega_{\Gamma,a}$ with finite perimeter, and let $\Gamma_S$ and $f_\pm$ be defined as above. We define then
\[
S^* := \Big\{\L(q,t)\in \Omega_{\Gamma,a}:\, q\in\Gamma_S,\, f_-(q)<t<f_+(q) \Big\}\,.
\]
\end{definition}
We can now show the main property of the set $S^*$, which will be fundamental for our purposes.

\begin{Lemma}[Area and perimeter of $S^*$]\label{Lem.technical}
Let $S$ be an open, bounded and connected subset of $\Omega$ of finite perimeter. Then
\begin{align*}
|S^*| \geq |S|\,, && P(S^*) \leq P(S) \,,
\end{align*}
and $f_\pm \in BV\big(\Gamma_S\big)$. Moreover, calling $f_\pm'\,dq$ the absolute continuous part of $Df_\pm$ and $D_s f_\pm$ its singular part, we have the validity of formula
\begin{equation}\label{stimaper}\begin{split}
P(S^*)&=\int_{\Gamma_S}\sqrt{\big(1-\kappa(q)\,f_+(q)\big)^2+f_+'(q)^2}\,dq+\int_{\Gamma_S}\sqrt{\big(1-\kappa(q)\,f_-(q)\big)^2+f_-'(q)^2}\, dq\\
&\hspace{60pt}+\big|D_s f_+\big|(\Gamma_S)+\big|D_s f_-\big|(\Gamma_S)+ \Big(f_+(q_l)-f_-(q_l)\Big)+\Big( f_+(q_r)-f_-(q_r)\Big)\,,
\end{split}
\end{equation}
where if $\Gamma_S$ is a subinterval of $\Gamma$ we denote by $q_l$ and $q_r$ its extremes, while if $\Gamma_S$ is compact the term $\Big(f_+(q_l)-f_-(q_l)\Big)+\Big( f_+(q_r)-f_-(q_r)\Big)$ has to be intended as $0$.
\end{Lemma}
\begin{proof}
First of all, the fact that $|S^*|\geq |S|$ is obvious, since by definition $S^* \supseteq S$. Concerning the inequality for the perimeter, we start by noticing that, by standard arguments, it is admissible to assume that $S$ is smooth. In fact, by the Compactness Theorem for BV functions (see for instance~\cite{AFP}), we can take a sequence $S_j$ of smooth sets converging in the $L^1$ sense to $S$ in such a way that $P(S_j)\to P(S)$. By definition, the corresponding sets $S^*_j$ converge to $S^*$, and by the lower semicontinuity of the perimeter this yields $P(S^*)\leq \liminf P(S^*_j)$. As an immediate consequence, once we establish the validity of this lemma for smooth sets, it will directly follow also in full generality.\par

The inequality $P(S^*)\leq P(S)$ for smooth sets is very easy to guess, but a bit boring to prove. For simplicity, we will divide the proof in some steps.
\step{I}{Non-intersecting curves cannot pass ``from above to below''.}
In this first step, we underline the following very easy topological fact. Here, by $\pi_1:\R^2\to \R$ we denote the first projection.
\begin{equation}\label{stepIold}\begin{array}{c}
\hbox{\emph{Let $q_0\in\R$, let $\gamma_1,\,\gamma_2\subseteq\R^2$ be two non-intersecting continuous curves in the plane such}}\\
\hbox{\emph{that $\min \pi_1\gamma_1=\min \pi_1\gamma_2 = q_0$. If $t_1:=\max\{t: (q_0,t)\in \gamma_1\}>\max\{t: (q_0,t)\in \gamma_2\}=:t_2$,}}\\
\hbox{\emph{then for all $q\in \pi_1\gamma_1\cap\pi_1\gamma_2$ one has $\max\{t: (q,t)\in \gamma_1\}>\max\{t: (q,t)\in \gamma_2\}$.}}
\end{array}\end{equation}
The meaning of this claim is very simple: if one has two continuous and non-intersecting curves in the plane, and the least abscissa of points in the two curves coincide (otherwise, it is obvious that the claim is false), then the curve which starts above always remains above.\par
To show the validity of the claim, suppose it is not true, and let $\bar q\in \pi_1\gamma_1\cap\pi_1\gamma_2$ be a point for which
\[
\bar t_1:= \max\{t: (\bar q,t)\in \gamma_1\}<\max\{t: (\bar q,t)\in \gamma_2\} =: \bar t_2
\]
(notice that the equality can not hold true, since the curves do not intersect).
Figure~\ref{Fig.mixed}(a) shows the situation.

\begin{figure}[h!]
\begin{center}
\begin{tabular}{ccc}
\includegraphics[width=0.47\textwidth]{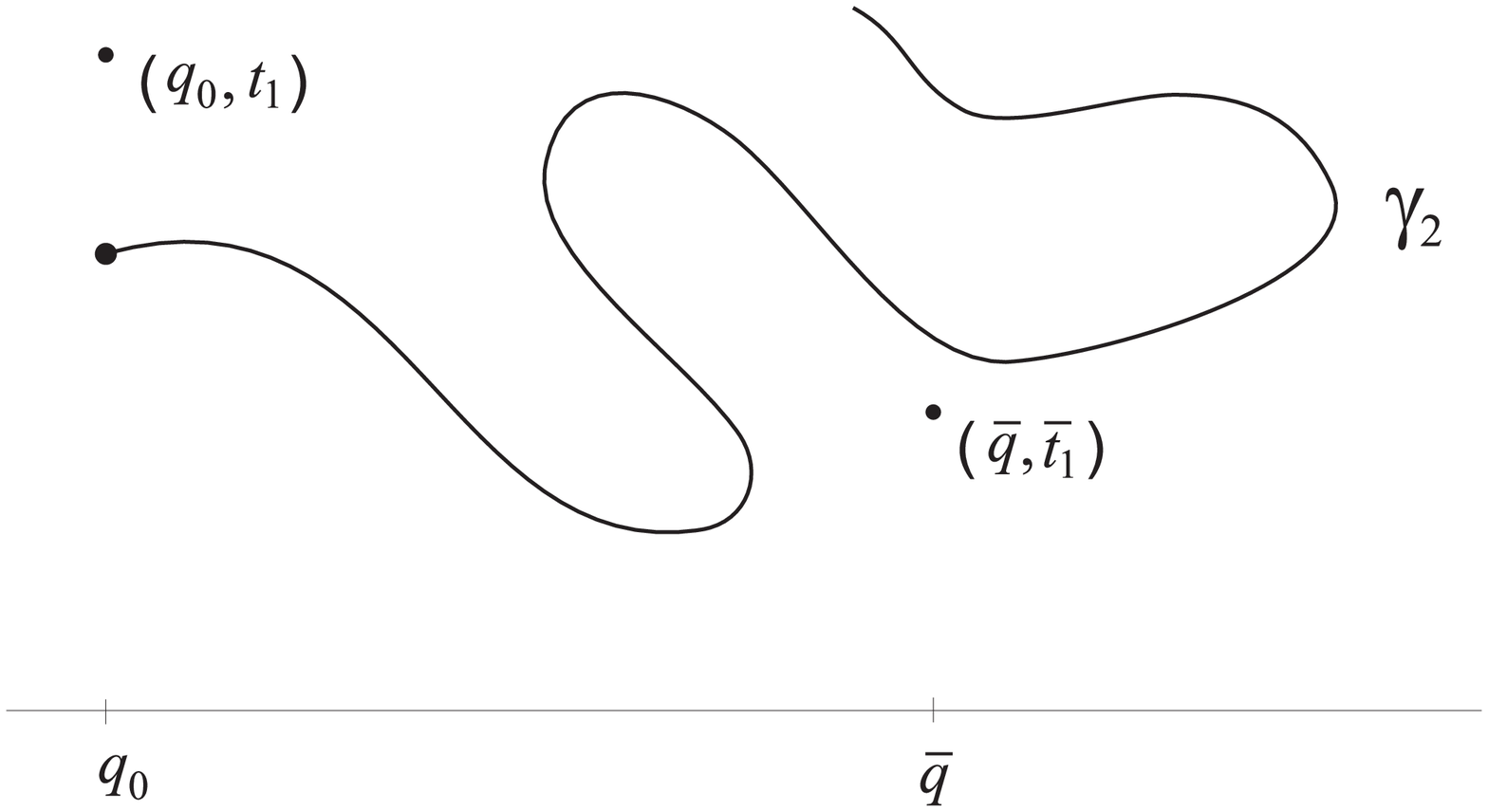}
& \qquad\quad
&\includegraphics[width=0.4\textwidth]{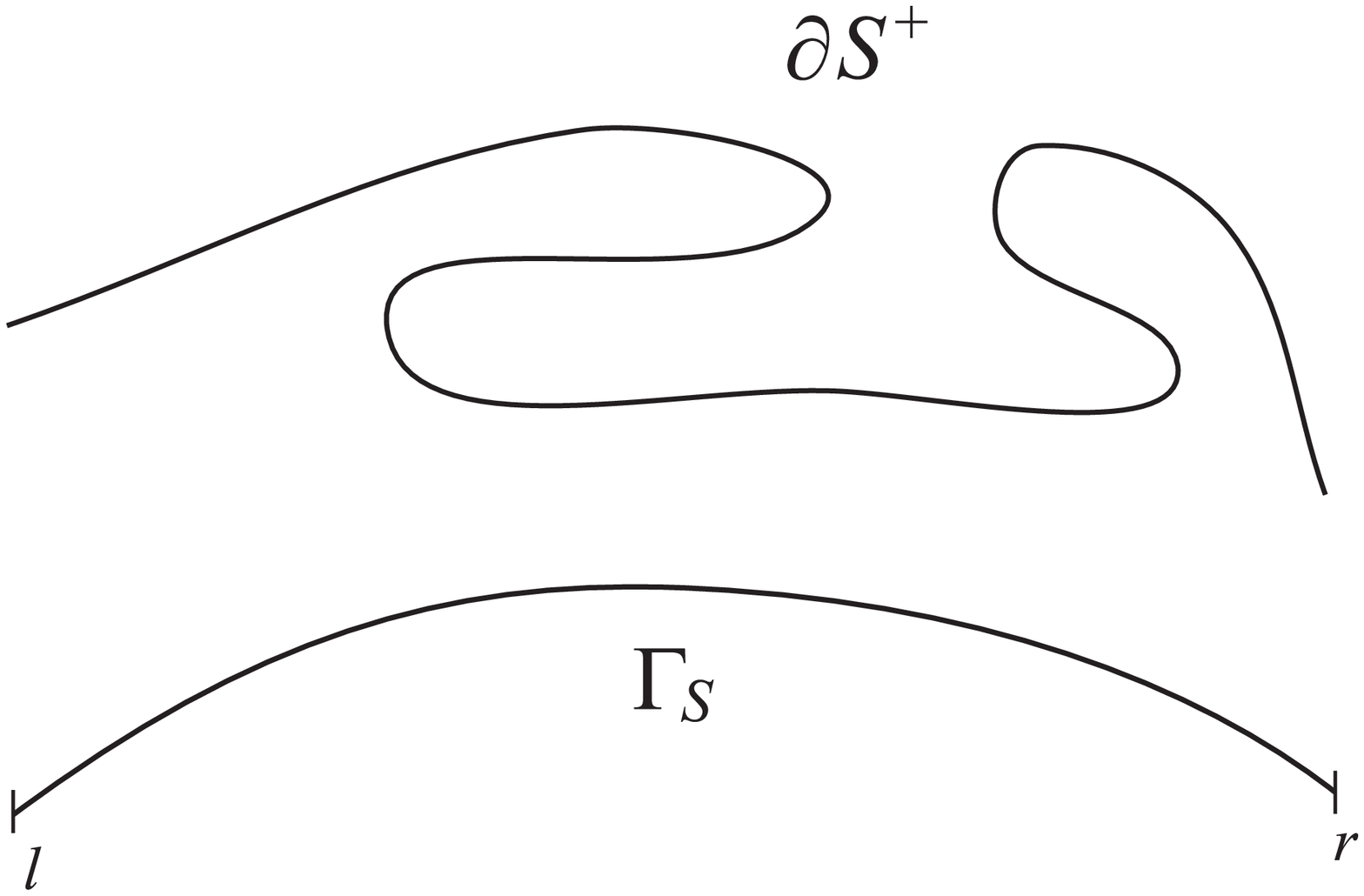}
\medskip \\
(a) The situation of Step~I.
&&
(b) A possible $\partial S^+$ in Step~III.
\end{tabular}
\end{center}
\caption{Figures clarifying some steps of the proof of Lemma~\ref{Lem.technical}.}
\label{Fig.mixed}
\end{figure}

\noindent
The curve $\gamma_1$, then, is contained by definition in
\[
A := \Big\{ (q,t) \in \R^2\setminus \gamma_2:\, q\geq q_0,\, \Big\} \setminus \Big\{ (\bar q, t) \in \R^2 :\, t > \bar t_1 \Big\} \,.
\]
This is a contradiction with the continuity of the curve $\gamma_1$, since the points $\big(q_0,t_1\big)$ and $\big(\bar q, \bar t_1 \big)$ are in $\gamma_1$ but belong to two distinct connected components of $A$. Therefore, the claim~(\ref{stepIold}) is proved.
\step{II}{First properties and some definitions.}
We can immediately observe some simple properties of $\partial S$ and give some related definitions. First of all, since $S$ is smooth, then $\partial S$ is the union of finitely many closed curves $\gamma_i,\, 1\leq i \leq N$; in particular, there is exactly one of them, say $\gamma_1$, which encloses the whole $S$.\par
Let us then consider separately the two cases whether or not $\Gamma_S$ is compact. If $\Gamma_S$ is not compact, thus it is an interval $(q_l,q_r)$, then it is immediate to observe that for every $q\in \Gamma_S$ the points $\L\big(q,f_\pm(q)\big)$ belong to $\gamma_1$. Let us then call $\partial S^+$ the part of $\gamma_1$ starting from $\L\big(l,f_+(l)\big)$, ending at $\L\big(r,f_+(r)\big)$ and containing $\L\big(q,f_+(q)\big)$ for every $q\in \Gamma_S$; similarly, we denote by $\partial S^-$ the part of $\gamma_1$ starting from $\L\big(r,f_-(r)\big)$, ending at $\L\big(l,f_-(l)\big)$ and containing $\L\big(q,f_-(q)\big)$ for every $q\in \Gamma_S$. An easy geometric argument ensures that $\partial S^+$ and $\partial S^-$
are well defined and do not intersect each other. We then obtain
\begin{equation}\label{4S}
\partial S \supseteq \gamma_1 = \partial S^+ \cup \partial S^r \cup \partial S^- \cup \partial S^l\,,
\end{equation}
being $\partial S^r$ (resp. $\partial S^l$) the part of $\gamma\setminus\big(\partial S^+ \cup \partial S^-\big)$ connecting $\L\big(r,f_+(r)\big)$ and $\L\big(r,f_-(r)\big)$ (resp. $\L\big(l,f_-(l)\big)$ and $\L\big(l,f_+(l)\big)$).\par

Consider now the case when $\Gamma_S$ is compact (thus, $\Gamma_S=\Gamma$). In this case, we can directly call $\partial S^+ = \gamma_1$, and again it is easy to observe that for every $q\in\Gamma$ one has $\L\big(q,f_+(q)\big) \in \partial S^+$. On the other hand, all the points $\L\big(q, f_-(q)\big)$ belong to a same connected component of $\partial S$ different from $\gamma_1$, say $\gamma_2$. We call then $\partial S^-=\gamma_2$ and $\partial S^r=\partial S^l = \emptyset$, so that also in this case~(\ref{4S}) holds true.\par

We conclude this step noticing that for the set $S^*$ the inclusion~(\ref{4S}) is in fact an equality by construction.

\step{III}{The ``upper boundary'' is well-ordered.}
In this step we show that the curve $\partial S^+$ reaches all the points $\L(q,f_+(q))$ in the ``correct order''.
This means that, if we parametrize $\partial S^+$ as $\gamma([0,1])$
with $\gamma(0)=\L(l,f_+(l))$ and $\gamma(1)=\L(r,f_+(r))$, then
\begin{equation}\label{stepIIIold}\begin{array}{c}
\hbox{\emph{If $\gamma(\sigma_1)=\L\big(q_1,f_+(q_1)\big)$ and $\gamma(\sigma_2)=\L\big(q_2,f_+(q_2)\big)$, one has $\sigma_1< \sigma_2 \Longleftrightarrow q_1<q_2$.}}
\end{array}\end{equation}
Notice that this fact is not trivial,
since the curve $\partial S^+$ does not have to be a graph on $\Gamma_S$,
hence it can, sometimes, move towards left, as in Figure~\ref{Fig.mixed}(b).
However, the figure itself suggests that the points $(q,f_+(q))$ are in any case reached ``from left to right''. Let us now show~(\ref{stepIIIold}). To do so, suppose by contradiction that it is not true. Hence, there exist $\sigma_1, \, \sigma_2,\, q_1$ and $q_2$ in such a way that $\gamma(\sigma_i)=\L(q_i,f_+(q_i))$ for $i=1,2$ but one has $\sigma_1>\sigma_2$ and $q_1<q_2$.
We can then give the following definitions, being $\pi$ the projection from $\Omega$ to $\Gamma$.
\[\begin{split}
\sigma_3 &= \min \Big\{ \sigma \in (\sigma_1,1): \pi(\gamma(\sigma)) = q_2\Big\} \,, \\
q^* &= \min \Big\{ \pi(\gamma(\sigma)): \sigma\in (\sigma_1,\sigma_3)\Big\} \,, \\
\sigma_0 &= \max \Big\{ \sigma\in (0,\sigma_2): \pi(\gamma(\sigma))=q^*\Big\}\,.
\end{split}\]
Notice that by construction one has $0<\sigma_0<\sigma_2<\sigma_1<\sigma_3<1$, as well as $q^*\leq q_1<q_2$. Now, consider the two curves $\gamma_1=\L^{-1}\big(\gamma_{|[\sigma_0,\sigma_2]}\big)$ and $\gamma_2=\L^{-1}\big(\gamma_{|[\sigma_1,\sigma_3]}\big)$, which are continuous and non-intersecting. Moreover, $\min\pi_1\gamma_1= \min \pi_1 \gamma_2 = q^*$, hence we can apply Step~I to derive that $\gamma_1$ is either ``always above'' or ``always below'' $\gamma_2$, in the sense of~(\ref{stepIold}). By checking $q=q_1$, one observes that $\gamma_1$ is below $\gamma_2$, since $\max\{\sigma: (q_1,\sigma)\in \gamma_2\}=f_+(q_1)$ is surely greater than $\max\{\sigma: (q_1,\sigma)\in \gamma_1\}$, by definition of $f_+$. On the other hand, by checking $q=q_2$, the very same reason shows that $\gamma_1$ is above $\gamma_2$, being $\max\{\sigma: (q_2,\sigma)\in \gamma_1\}=f_+(q_2)$. The contradiction shows the validity of~(\ref{stepIIIold}), hence this step is concluded.
\step{IV}{The functions $f_\pm$ are in $BV(\Gamma_S)$.}
Let us fix an arbitrary $N\in\N$, and an arbitrary sequence $l=q_0<q_1< \cdots < q_N <q_{N+1} = r$ in $\Gamma_S$. We claim that
\begin{equation}\label{stepIVold}
\sum_{i=0}^N \big| f_+(q_i) - f_+(q_{i+1}) \big| \leq \H^1\big( \partial S^+ \big)\,,
\end{equation}
being $\H^1$ the Hausdorff measure of dimension $1$, that is, the length. Notice that this inequality would show that $f_+\in BV(\Gamma_S)$, since $S$ is of finite perimeter.\par
To show the estimate, let us call $\gamma_i$ the part of the curve $\partial S^+$ which connects $\L(q_i,f_+(q_i))$ with $\L(q_{i+1},f_+(q_{i+1}))$. By the preceding steps, we know that $\partial S^+$ consists of the disjoint union of the curves $\gamma_i$, so that
\[
\H^1 \big( \partial S^+ \big) = \sum_{i=0}^N \H^1 \big(\gamma_i\big)\,.
\]
Hence, (\ref{stepIVold}) will follow at once as soon as we observe that for any $i=0, \,\dots\,, \, N$ one has
\begin{equation}\label{showthenfollows}
\H^1\big(\gamma_i\big) \geq \Big| \L\big( q_i,f_+(q_i)\big)-\L\big( q_{i+1},f_+(q_{i+1})\big) \Big| > \big| f_+(q_i) - f_+(q_{i+1}) \big|\,.
\end{equation}
The first inequality is trivial, since it just says that the length of the curve $\gamma_i$ is greater than the distance of its extreme points. Concerning the strict inequality, instead, let us denote for brevity
\begin{align*}
P:= \L \big(q_i, f_+(q_i) &\big)\,, \qquad
Q:= \L \big(q_{i+1}, f_+(q_{i+1}) \big)\,, \qquad
Q':= \L \big(q_i, f_+(q_{i+1}) \big)\,, \\
&S':= \L\big(q_i,0 \big)\,, \qquad\qquad\qquad
S:= \L\big(q_{i+1},0\big)\,.
\end{align*}
Hence, assuming that $f_+(q_i)\geq f_+(q_{i+1})\geq 0$ (it is then trivial to modify the argument to cover the other cases), one has
\[
\overline{PQ'}+\overline{Q'S'} = \overline{PS'} < \overline{PS} < \overline{PQ} + \overline{QS} = \overline{PQ} + \overline{Q'S'}\,,
\]
where the first inequality is due to the fact that, by definition, $S'$ is the closest point to $P$ inside $\Gamma$. The inequality above says that $\overline{PQ'} < \overline{PQ}$, which is precisely the missing inequality in~(\ref{showthenfollows}). As explained above, this implies the validity of~(\ref{stepIVold}), hence the fact that $f_+\in BV(\Gamma_S)$.\par
Of course, the very same argument shows that also $f_-\in BV(\Gamma_S)$.
\step{V}{One has $\H^1\big(\partial S^+\big)\geq \H^1\big(\partial {S^*}^+\big)$.}
Let us define $\{q_i,\, i\in \N\}\subseteq \Gamma_S$ the jump points of $f_+$, which are countably many since $f_+\in BV(\Gamma_S)$. For any $i$, moreover, let us call
\begin{align*}
f_+^l(q_i) = \lim_{q\uparrow q_i} f_+(q)\,, &&
f_+^r(q_i) = \lim_{q\downarrow q_i} f_+(q)\,:
\end{align*}
being $f_+\in BV(\Gamma_S)$, these two limits exist and correspond to the $\liminf$ and the $\limsup$ of $f_+$ for $q\to q_i$. In particular, one has that
\[
\partial \big( {S^*}^+\big) = \Big\{ \L\big(q,f_+(q)\big):\, q\in \Gamma_S \Big\} \cup \bigcup_{i\in\N} J_i\,,
\]
where $J_i$ is the segment joining $\L\big(q_i,f_+^l(q_i)\big)$ and $\L\big(q_i,f_+^r(q_i)\big)$. Let us fix now $\eps>0$, so that there exists $N\in \N$ such that
\[
\sum_{i> N} \big| J_i \big| < \eps\,.
\]
For simplicity, we can assume that the points $q_i$ are ordered so that $l < q_1 < \cdots < q_n < r$. We can now pick, for any $1\leq i\leq N$, two points $q_i^l<q_i < q_i^r$ in $\Gamma_S$ in such a way that
\begin{itemize}
\item the different intervals $(q_i^l,q_i^r)$ are disjoint;
\item for any $i$ one has
\[
\big|f_+(q_i^l) - f_+^l(q_i) \big| + \big|f_+(q_i^r) - f_+^r(q_i) \big| \leq \frac \eps N\, ;
\]
\item one has
\[
\H^1\Big( \big( \partial {S^*}^+\big) \cap \L \big( (q_i^l,q_i^r) \times (-a,a)\big)\Big)
\leq \big| J_i \big| + \frac \eps N
= \big| f_+^l(q_i) - f_+^r(q_i) \big| + \frac \eps N\,.
\]
\end{itemize}
Now, we can consider the ``bad'' intervals $B_i = (q_i^l, q_i^r)$, where there are high jumps, and the ``good'' intervals $G_i= (q_i^r, q^l_{i+1})$, where there are not. Define also $G_0=(l,q_1^l)$, while $G_N=(q_N^r,r)$. Therefore, we have decomposed $\Gamma_S=\cup_{i\leq N} B_i \cup G_i$. For any good interval $G_i$, one has
\[
\partial {S^*}^+ \cap \L \big( G_i \times (-a,a) \big) = \Big\{ \L\big(q,f_+(q)\big):\, q\in G_i \Big\} \cup \bigcup_{j\in\N} \widetilde J_{i,j}\,,
\]
where $\widetilde J_{i,j}$ are the jumps of $f_+$ contained in the interval $G_i$. Of course all the jumps $\widetilde J_{i,j}$, varying $0\leq i\leq N$ and $j\in\N$, correspond to different jumps $J_i$ for $i>N$. For any bad interval $B_i$, moreover, call $\gamma_i$ the part of the curve $\partial S^+$ from $\L\big(q_i^l,f_+(q_i^l)\big)$ to $\L\big(q_i^r,f_+(q_i^r)\big)$. Thanks to Step~III, all the curves $\gamma_i$ are disjoint, and in particular $\L(q,f_+(q))$ belongs to $\gamma_i$ if and only if $q\in B_i$. Since we know that $\L(q,f_+(q))\in \partial S^+$ for all $q\in\Gamma_S$, this implies that
\[
\H^1\big( \partial S^+\big) \geq \H^1\Big(\big\{\L(q,f_+(q)):\, q\in \bigcup\nolimits_{i=0}^N G_i \big\} \Big) +\sum_{i=1}^N \H^1\big(\gamma_i\big)\,.
\]
Notice also that, as shown with~(\ref{showthenfollows}) in Step~IV, one has for each $1\leq i\leq N$ that
\[
\H^1\big( \gamma_i \big) > \big| f_+(q^l_i) - f_+(q^r_i)\big|\,.
\]
So, we can finally conclude, using all the properties listed above, that
\[\begin{split}
\H^1\big( \partial {S^*}^+\big) & =
\sum_{i=0}^N \H^1 \Big( \partial {S^*}^+\cap\L \big( G_i \times (-a,a)\big) \Big)+\sum_{i=1}^N \H^1 \Big( \partial {S^*}^+\cap\L \big( B_i \times (-a,a)\big) \Big)\\
&\leq\sum_{i=0}^N \bigg(\H^1 \Big(\big\{\L\big(q,f_+(q)\big):\, q\in G_i\big\}\Big)+\sum_{j\in\N} \big| \widetilde J_{i,j}\big|\bigg)
+\sum_{i=1}^N \bigg( \big| f_+^l(q_i) - f_+^r(q_i) \big| + \frac \eps N\bigg)\\
&\leq  \H^1 \Big(\big\{\L\big(q,f_+(q)\big):\, q\in \bigcup\nolimits_{i=0}^N G_i\big\}\Big) + \sum_{i>N} \big| J_i\big|
+\sum_{i=1}^N \bigg( \big| f_+(q^l_i) - f_+(q^r_i) \big| + 2\,\frac \eps N\bigg)\\
&\leq  \H^1 \Big(\big\{\L\big(q,f_+(q)\big):\, q\in \bigcup\nolimits_{i=0}^N G_i\big\}\Big) + \eps
+\sum_{i=1}^N \bigg( \H^1\big( \gamma_i\big) + 2\,\frac \eps N\bigg)\\
&\leq \H^1\big( \partial S^+\big) + 3 \eps\,.
\end{split}\]
Since $\eps>0$ was arbitrary, this step is concluded.
\step{VI}{Conclusion.}
By~(\ref{4S}), we know that
\[
\partial S \supseteq \partial S^+ \cup \partial S^- \cup \partial S^l \cup \partial S^r\,,
\]
and the union is disjoint. On the other hand, as noticed at the end of Step~II, we have
\[
\partial S^* = \partial {S^*}^+ \cup \partial {S^*}^- \cup \partial {S^*}^l \cup \partial {S^*}^r\,.
\]
By Step~V we know that $\H^1\big(\partial S^+\big)\geq \H^1\big(\partial {S^*}^+\big)$, and in the very same way of course $\H^1\big(\partial S^-\big)\geq \H^1\big(\partial {S^*}^-\big)$. Let us then focus for a moment on $\partial S^l$ and on $\partial {S^*}^l$. If $\Gamma_S$ is compact, they are both empty. Otherwise, $\partial S^l$ is a curve between $\L(l,f_-(l))$ and $\L(l,f_+(l))$, while $\partial {S^*}^l$ is the segment joining the same points. As a consequence, in any case one has $\H^1\big(\partial S^l\big)\geq \H^1\big(\partial {S^*}^l\big)$. Of course, similarly $\H^1\big(\partial S^r\big)\geq \H^1\big(\partial {S^*}^r\big)$. Adding up the four inequalities, we finally get that $\H^1\big(\partial S\big) \geq \H^1\big(\partial S^*\big)$.\par
Concerning formula~(\ref{stimaper}), it is immediate to obtain it for smooth functions $f_-$ and $f_+$, while the generalization for $BV$ functions is standard.
\end{proof}
With the above result at hand, it will be quite easy to obtain the lower bound.
\begin{Lemma}[Lower bound]\label{Lem.lower}
For a curved strip $\Omega_{\Gamma,a}$ of any kind, one has
\[
h(\Omega_{\Gamma,a}) \geq \frac{1}{a}\,.
\]
Moreover, if the inequality above is an equality and there is a Cheeger set, then this Cheeger set must be $\Omega_{\Gamma,a}$ itself.
\end{Lemma}
\begin{proof}
Let $S$ be any open connected set of finite perimeter in $\Omega_{\Gamma,a}$, and let $S^*$ be as in Definition~\ref{defS*}. Denoting by
\begin{align*}
t_-:= \inf  \big\{ f_-(q) :\, q\in\Gamma_S \big\}\,, &&
t_+:= \sup  \big\{ f_+(q):\, q\in\Gamma_S \big\}\,,
\end{align*}
we can easily estimate
\begin{equation}\label{exp}\begin{split}
|S^*| & = \int_\Gamma \int_{f_-(q)}^{f_+(q)} \big(1-\kappa(q)\,t\big) \, dt \, dq= \int_\Gamma \big(f_+(q)-f_-(q)\big)
\left(1-\kappa(q)\,\frac{f_+(q)+f_-(q)}{2}\right) dq\\
& \leq (t_+-t_-) \int_\Gamma \bigg(1-\kappa(q)\,\frac{f_+(q)+f_-(q)}{2}\bigg)\, dq \,.
\end{split}\end{equation}
On the other hand, by~(\ref{stimaper}) it is easy to estimate the perimeter of $S^*$ as
\[\begin{split}
P(S^*)&=\int_{\Gamma_S}\sqrt{\big(1-\kappa(q)\,f_+(q)\big)^2+f_+'(q)^2}\, dq+\int_{\Gamma_S}\sqrt{\big(1-\kappa(q)\,f_-(q)\big)^2+f_-'(q)^2}\,dq\\
&\hspace{60pt}+ \big|D_s f_+\big|(\Gamma_S)+\big|D_s f_-\big|(\Gamma_S)+ \Big( f_+(q_0)-f_-(q_0)\Big)+\Big( f_+(q_1)-f_-(q_1)\Big)\,,\\
&\geq 2 \int_{\Gamma_S} \bigg(1-\kappa(q)\,\frac{f_+(q)+f_-(q)}{2}\bigg)\,dq\,,
\end{split}\]
simply by neglecting both the absolutely continuous and the singular part of $Df$. Hence, thanks to Lemma~\ref{Lem.technical} we can readily deduce that
\[
\frac{P(S)}{|S|} \geq \frac{P(S^*)}{|S^*|} \geq \frac{2}{\ t_+-t_-} \geq \frac{1}{a} \,,
\]
where the last inequality is due to the trivial bounds $-a \leq t_- < t_+ \leq a$. Finally, if $h(\Omega_{\Gamma,a})=1/a$ and there is some Cheeger set $\C=\C_{\Omega_{\Gamma,a}}$, then all the preceding inequalities must be equalities for $S=\C$, from which it immediately follows that $f_+$ and $f_-$ are constant, and that $t_\pm = \pm a$, thus $\C=\Omega_{\Gamma,a}$.
\end{proof}

\begin{Remark}{\rm
As a consequence of~\eqref{Ch.bound} for~$p=2$, from the above result we get the lower bound
\[
\lambda_2(\Omega_{\Gamma,a}) \geq \frac{1}{4a^2} \,,
\]
which is in fact weaker that the bound
\[
\lambda_2(\Omega_{\Gamma,a}) \geq \frac{j_{0,1}^2}{4a^2}
\]
known from~\cite{EFK}. Here $j_{0,1} \approx 2.4$ denotes the first positive zero of the Bessel function~$J_0$. In fact, even a better bound, reflecting the local geometry of~$\Gamma$ and valid in arbitrary dimensions, is established in~\cite{EFK}.}
\end{Remark}

\begin{Remark}\label{Rem.test}{\rm
It is possible to establish the lower bound of Lemma~\ref{Lem.lower}
directly from Theorem~\ref{Thm.Grieser},
without using the ``stripization'' procedure $S^*$
of Definition~\ref{defS*} and its properties stated in Lemma~\ref{Lem.technical}.
Indeed, inspired by the formula of
\cite[Sec.~11, Ex.~4]{Bellettini-Caselles-Novaga_2002}
for the annulus, let us introduce the function $V_t:\Gamma\times(-a,a)\to\R$ by
\[
  V_t(q,t) :=
  \left\{
  \begin{aligned}
  &\frac{(1-\kappa(q)\,a)\,(1+\kappa(q)\,a)-(1-\kappa(q)\,t)^2}
  {2 \, a \, \kappa(q) \, (1-\kappa(q)\,t)}
  && \mbox{if} \quad \kappa(q)\not=0 \,,
  \\
  &\frac{t}{a}
  && \mbox{if} \quad \kappa(q)=0 \,.
  \end{aligned}
  \right.
\]
Note that the value for vanishing curvature corresponds to
taking the limit $\kappa(q) \to 0$ in the formula for positive curvatures.
One easily checks that the vector field $V(q,t):=(0,V_t(q,t))$,
where the components are considered with respect to the coordinates~$(q,t)$,
satisfies $\|V\|_{L^\infty(\Gamma\times(-a,a))} = 1$ and
\[
(\div V)(q,t)
= \frac{1}{1-\kappa(q)\,t} \, \partial_t \big[(1-\kappa(q)\,t) \, V_t(q,t)\big]
= \frac{1}{a}
\]
for every $(q,t) \in \Gamma\times(-a,a)$.
Hence, the searched lower bound is a consequence of Theorem~\ref{Thm.Grieser}.
However, Lemma~\ref{Lem.technical} is needed to establish
some finer properties of the Cheeger constant and Cheeger set.}
\end{Remark}

\section{The main results}

This section is devoted to show our two main results, namely Theorem~\ref{Thm.Cheeger}, which deal with the case of curved annuli or not finite curved strips, and Theorem~\ref{Thm.bounded},
which deals with finite curved strips.

\subsection{The case of a curved annulus and that of a not finite curved strip}

\begin{Theorem}\label{Thm.Cheeger}
Let $\Gamma$ be compact, infinite or semi-infinite. Then
\begin{equation}\label{strip.complete}
h(\Omega_{\Gamma,a}) = \frac{1}{a} \,.
\end{equation}
In particular:
\begin{enumerate}
\item[\emph{(i)}]
If~$\Gamma$ is compact (\emph{i.e.}~$\Omega_{\Gamma,a}$ is a curved annulus),
then the infimum of~\eqref{Ch.constant}
is attained and the unique Cheeger set is $\C_{\Omega_{\Gamma,a}}=\Omega_{\Gamma,a}$.
\item[\emph{(ii)}]
If~$\Gamma$ is infinite or semi-infinite (\emph{i.e.}~$\Omega_{\Gamma,a}$ is an infinite or semi-infinite curved strip), then the infimum of~\eqref{Ch.constant}
is not attained, but the sequence $\Omega_{\Gamma_L,a}$ of Lemma~\ref{Lem.upper} is an optimizing sequence for $L\to \infty$.
\end{enumerate}
\end{Theorem}
\begin{proof}
The equality~\ref{strip.complete} follows directly from the upper estimate of Lemma~\ref{Lem.upper} and the lower estimate of Lemma~\ref{Lem.lower}.\par
From the characterization of Lemma~\ref{Lem.lower}, moreover, we know that the unique possible Cheeger set is the whole $\Omega_{\Gamma,a}$. Since this set has an infinite area and perimeter in the case of an infinite or semi-infinite curved strip, we get the non-existence result of a minimizer for the case~(ii), while the fact that $\Omega_{\Gamma_L,a}$ is a minimizing sequence for $L\to \infty$ follows by Lemma~\ref{Lem.upper}. On the other hand, in case~(i) we know by compactness that some Cheeger set must exist, hence the existence and uniqueness of the whole $\Omega_{\Gamma,a}$ as a Cheeger set again come by Lemma~\ref{Lem.lower}.
\end{proof}

\subsection{The case of a finite curved strip}

\begin{Theorem}\label{Thm.bounded}
Let~$\Gamma$ be non-complete and finite (hence, $\Omega_{\Gamma,a}$ is a finite curved strip). Then there exists a positive universal constant $c$ such that
\begin{equation}\label{strip.bounded}
\frac{1}{a} + \frac{c}{|\Gamma|}
\leq h(\Omega_{\Gamma,a}) \leq
\frac{1}{a} + \frac{2}{|\Gamma|} \,.
\end{equation}
For instance, one may take $c=1/400$. Moreover, the infimum in~\eqref{Ch.constant} is attained for some connected set $\C_{\Omega_{\Gamma,a}}\subsetneq \Omega_{\Gamma,a}$.
\end{Theorem}
\begin{proof}
Concerning the existence of a Cheeger set $\C=\C_{\Omega_{\Gamma,a}}$, and in particular of a connected one, this follows by Theorem~\ref{genprop}. From the same Theorem, we know also that $\partial \C\cap \Omega_{\Gamma,a}$ is made by arcs of circle of radius $h(\Omega_{\Gamma,a})^{-1}$, and it cannot coincide with the whole set $\Omega_{\Gamma,a}$ again by Theorem~\ref{genprop}, since $\C$ may not have corners. As a consequence, by the characterization of Lemma~\ref{Lem.lower} we deduce that $h(\Omega_{\Gamma,a})>1/a$. To conclude, we have then only to give a proof of the bounds~(\ref{strip.bounded}), which will be done in some steps.
\step{I}{The upper bound.}
Obtaining the upper bound is very easy: it is enough to remind that
\begin{align*}
P\big(\Omega_{\Gamma,a}\big)= 2 |\Gamma| + 4a\,, && \big|\Omega_{\Gamma,a}\big|= 2 a\big|\Gamma\big|\,,
\end{align*}
as already checked for instance in~(\ref{approximation}), and then
\[
h(\Omega_{\Gamma,a}) \leq \frac{P\big(\Omega_{\Gamma,a}\big)}{\big|\Omega_{\Gamma,a}\big|}
=\frac{2 |\Gamma| + 4a}{2a\big|\Gamma\big|}
=\frac{1}{a} + \frac{2}{|\Gamma|}
\,.
\]
\step{II}{The lower bound: behaviour of the arcs of $\partial\C \cap \Omega_{\Gamma,a}$.}
Thanks to Theorem~\ref{genprop}, we know that $\partial\C$ can not have corners. Hence, $\partial \C\cap \Omega_{\Gamma,a}$ is not empty, and it is done by some arcs of circle, all of radius $1/h(\Omega_{\Gamma,a})$, hence strictly smaller than $a$ as noticed above, such that all the four corners of $\Omega_{\Gamma,a}$ are ruled out from $\C$. Denoting by $q_0$ and $q_1$ the extreme points of $\Gamma$, let us call for simplicity ``up'', ``down'', ``left'' and ``right'' the four parts of $\partial \Omega_{\Gamma,a}$ given by the points of the form $\L(q,a)$, $\L(q,-a)$, $\L(q_0,t)$ and $\L(q_1,t)$ for $q\in\Gamma$ and $t\in (-a,a)$ respectively. We aim to show the following claim:
\begin{equation}\label{stepII}\begin{array}{c}
\hbox{\emph{All the arcs of circle of $\partial \C\cap\Omega_{\Gamma,a}$ connect two points of $\Omega_{\Gamma,a}$,}}\\
\hbox{\emph{at least one of which is either in the left or in the right part.}}
\end{array}\end{equation}
To show this claim, we have to exclude the case of an arc of circle starting and ending in the upper part, and the case of an arc connecting the up and the down (the case of an arc starting and ending in the bottom part is exactly the same as the first one).\par
Suppose first that there is an arc of circle connecting the points $P$ and $Q$, both in the upper part, thus $P=\L(q',a)$ and $Q=\L(q'',a)$. By Theorem~\ref{genprop}, we know that the circle is tangent to $\partial \Omega_{\Gamma,a}$ at $P$ and $Q$, hence its centre $O$ is the intersection between the two lines which are normal to $\partial\Omega_{\Gamma,a}$ at $P$ and $Q$, which are $t\mapsto \L(q',t)$ and $t\mapsto \L(q'',t)$. Since the radius $r$ of these circles is at most $a$, the two lines must intersect in the point $\L(q',a-r)\equiv \L(q'',a-r)$, while this is impossible for any $r\leq 2a$ because $\L$ is one-to-one.\par
A very similar argument works assuming that an arc of circle connects the point $P=\L(q',a)$ in the upper part with the point $Q=\L(q'',-a)$ in the lower part. Indeed, again the circle would be tangent to $\partial\Omega_{\Gamma,a}$ at both $P$ and $Q$, so that its centre would be in the intersection of the segments $t\mapsto \L(q',t)$ and $t\mapsto \L(q'',t)$. This is impossible if the circle has radius smaller than $2a$ for $q'\neq q''$, but it is impossible also for a radius strictly smaller than $a$ in the case $q'=q''$. This completely shows~(\ref{stepII}).\par
Notice now that by definition the left and the right part of $\partial\Omega_{\Gamma,a}$ are segments, so also the case of a circle starting and ending in the left is impossible, as well as an arc starting and ending in the right. As a conclusion, we now know that there can be either $2$, or $3$, or $4$ arcs of circle in $\partial\C$. The simplest case is when there are four arcs, each of which making a ``rounded corner''. This happens for instance for a rectangle (\emph{i.e.}, if $\Gamma$ is a segment), and more in general if $a$ is sufficiently small with respect to $\big|\Gamma\big|$. However, it is also possible that there are only three arcs, one of which connecting the left and the right part of the boundary. This happens for instance whenever the upper or the lower part of the boundary are very short due to a big (but still admissible) curvature of $\Gamma$. An example of this situation is a sector of an annulus with very small inner radius, which then is very similar to a triangle: in this case the boundary of $\C$ does not touch the inner circle (some examples of this kind are shown in the next section). Concerning the last possibility, namely only two arcs of circle both connecting left and right, we have no example in mind and it is maybe impossible,
but we do not need to exclude this case within this proof. Indeed, in Steps~III and~IV we will show the Theorem in the case of the four rounded corners, while in the last Step~V we will show how it is always possible to reduce to this case.

\step{III}{The lower bound: the case when $\C$ has four rounded corners, statement of the Claim~(\ref{claim}).}
To show the lower bound, we start from the case when $\C$ has four rounded corners. Let us recall that, as shown by~(\ref{stimaper}) in Lemma~\ref{Lem.technical}, one has
\begin{equation}\label{stimaper2}\begin{split}
P(\C) &= \int_\Gamma \sqrt{\big(1-\kappa(q)\,f_+(q)\big)^2+f_+'(q)^2} \, dq + \int_\Gamma \sqrt{\big(1-\kappa(q)\,f_-(q)\big)^2+f_-'(q)^2} \, dq\\
&\hspace{60pt}+ \big|D_s f_+\big|(\Gamma)+\big|D_s f_-\big|(\Gamma)+ \Big( f_+(q_0)-f_-(q_0)\Big)+\Big( f_+(q_1)-f_-(q_1)\Big)\,,
\end{split}
\end{equation}
where $f_\pm'\,dq$ is the absolute continuous part of $Df_\pm$ and $D_s f_\pm$ its singular part (notice that, in the language of Definition~\ref{defS*}, we have $\Gamma_{\C}=\Gamma$ thanks to Step~II). Hence, in particular
\[
P(\C) \geq \int_\Gamma \bigg( 2-\kappa(q)\Big(f_+(q)+f_-(q)\Big) \bigg)\, dq\,.
\]
We claim that, at least in the case when $\C$ has the four corners,
\begin{equation}\label{claim}
P(\C) \geq \int_\Gamma \bigg(2-\kappa(q)\Big(f_+(q)+f_-(q)\Big)\bigg) \, dq + \frac{1}{50} \,a\,.
\end{equation}
We will prove this estimate in next step, now we show how this implies the thesis. In fact, we can easily estimate, as in~(\ref{exp}), the area of $\C$ as
\[\begin{split}
\big|\C\big| &= \int_\Gamma \int_{f_-(q)}^{f_+(q)} \big(1-\kappa(q)\,t\big) \, dt \, dq
= \int_\Gamma \Big(f_+(q)-f_-(q)\Big) \bigg( 1 -\kappa(q)\,\frac{f_+(q)+f_-(q)}{2}\bigg)\,dq\\
&\leq a  \int_\Gamma \bigg(2-\kappa(q)\,\Big(f_+(q)+f_-(q)\Big) \bigg)dq\,.
\end{split}\]
Hence, using~(\ref{claim}) we get~(\ref{strip.bounded}) because,
recalling that $a \|\kappa\|_{L^\infty(\Gamma)} \leq 1$ (as pointed out in Section~\ref{GP})
\begin{equation}\label{cp}\begin{split}
h(\Omega_{\Gamma,a}) &= \frac{P\big(\C\big)}{\big|\C\big|}
\geq \frac{\bal \int_\Gamma \bigg(2-\kappa(q)\Big(f_+(q)+f_-(q)\Big)\bigg) \, dq + \frac{1}{50} \,a\eal}{\bal a  \int_\Gamma \bigg(2-\kappa(q)\,\Big(f_+(q)+f_-(q)\Big)\bigg)\, dq\eal}\\
&= \frac{1}{a}+ \frac{1}{50\bal  \int_\Gamma \bigg(2-\kappa(q)\Big(f_+(q)+f_-(q)\Big) \bigg)\, dq\eal}
\geq \frac{1}{a}+ \frac{1}{50\big|\Gamma\big|\big(2+2 a \|\kappa\|_{L^\infty(\Gamma)} \big)}
\geq \frac{1}{a} + \frac{1}{200\,|\Gamma|}\,.
\end{split}\end{equation}

\step{IV}{The lower bound: the case when $\C$ has four rounded corners, proof of Claim~(\ref{claim}).}
Here we show that, assuming that $\partial\C\cap \Omega_{\Gamma,a}$ consists of four arcs of circle, the claim~(\ref{claim}) holds. This will be done by considering a single arc. To choose it, we start noticing that~(\ref{stimaper2}) already trivially implies~(\ref{claim}) if $f_+(q_1) - f_-(q_1) \geq a/50$. As a consequence, we can assume that
\begin{equation}\label{upperright}
f_+(q_1) \leq \frac{1}{100}\, a
\end{equation}
and we concentrate on the arc of circle corresponding to the ``upper right corner''. Of course, if~(\ref{upperright}) were not true, then one could assume $f_-(q_1)\geq - a/100$ and then make the completely symmetric considerations on the ``lower right corner''. As shown in Figure~\ref{figTHREE}, we call $\gamma$ the arc of circle that we are considering, and we can also look $\gamma$ in the reference rectangle, where of course it is no more part of a circle. We will call $\eta$ as in the Figure.\par
\begin{figure}[htbp]
\begin{center}
\includegraphics[width=0.99\textwidth]{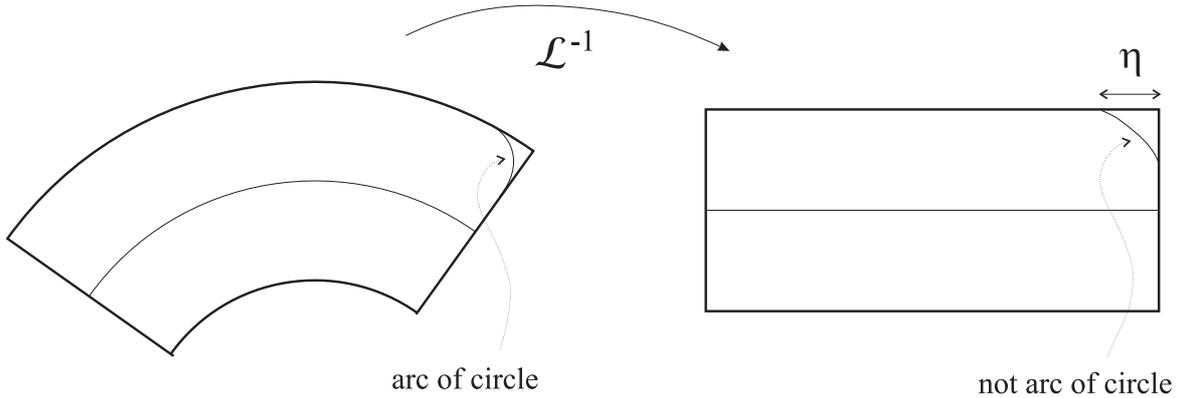}
\end{center}
\caption{The situation (both in $\Omega$ and in the reference configuration) of Step~IV.}
\label{figTHREE}
\end{figure}

Calling $\Gamma_\gamma$ the part of $\Gamma$ related to the curve $\gamma$, hence the subset of $\Gamma$ such that
\[
\gamma = \Big\{ \big(q,f_+(q)\big):\, q\in \Gamma_\gamma \Big\}\,,
\]
we can subdivide $\Gamma_\gamma$ in two parts, namely
\begin{align*}
\Gamma_1:= \bigg\{ q\in\Gamma_\gamma:\, \big|f_+'(q)\big|<\frac 15 \bigg\} \,, && \Gamma_2:= \bigg\{ q\in\Gamma_\gamma:\, \big|f_+'(q)\big|\geq \frac15 \bigg\}\,.
\end{align*}
Notice that the above subdivision makes sense because $f_+$ has no singular part inside $\Gamma_\gamma$ (since the image of its graph under $\L$ is an arc of circle). By definition,
\begin{equation}\label{1pic}
\int_{\Gamma_1} \big|f_+'(q)\big|\, dq \leq \frac{\big| \Gamma_1 \big|}{5} \leq \frac{\eta}{5} \leq \frac 25 \,a\,.
\end{equation}
In the last inequality we used that $\eta\leq 2a$, which in turn is true because $\gamma$ is an arc of circle or radius smaller than $a$ (keep in mind that by Lemma~\ref{Lem.lower} we already know that
$h(\Omega_{\Gamma,a})\geq 1/a$) and the lengths in the reference rectangle are at most the double of the true lengths (recall also that the arcs of circle touch $\partial\Omega_{\Gamma,a}$ tangentially, as we know by Theorem~\ref{genprop}). But then, thanks to~(\ref{upperright}), one has
\[
\int_{\Gamma_\gamma} \big|f_+'(q)\big| \, dq  \geq \frac{99}{100}\,a\,,
\]
so that by~(\ref{1pic}) we get
\begin{equation}\label{use1}
\int_{\Gamma_2} \big|f_+'(q)\big|\, dq \geq \bigg(\frac{99}{100}-\frac 25\bigg)\,a=\frac{59}{100}\,a\,.
\end{equation}
Recalling again that $0<1-\kappa(q)f_+(q) < 2$, a trivial calculation ensures that for any $q\in\Gamma_2$ it is
\begin{equation}\label{use2}
\sqrt{\big(1-\kappa(q)\,f_+(q)\big)^2+f_+'(q)^2} \geq \big(1-\kappa(q)\,f_+(q)\big) +\frac {1}{25}\, \big|f_+'(q)\big|\,.
\end{equation}
Hence, thanks to~(\ref{use1}) and~(\ref{use2}) we can estimate the length of $\gamma$ as
\[\begin{split}
\big| \gamma \big| &= \int_{\Gamma_\gamma} \sqrt{\big(1-\kappa(q)\,f_+(q)\big)^2+f_+'(q)^2} \, dq
\geq \int_{\Gamma_\gamma} \big(1-\kappa(q)\,f_+(q)\big)\, dq +\frac 1{25} \int_{\Gamma_2} \big| f_+'(q) \big|\,dq\\
&\geq \int_{\Gamma_\gamma} \big(1-\kappa(q)\,f_+(q)\big)\, dq +\frac{59}{25\cdot 100}\,a\,.
\end{split}\]
Recalling now formula~(\ref{stimaper2}) for the perimeter of $\C$, we finally conclude
\[
P(\C) \geq \int_\Gamma \bigg(2-\kappa(q)\Big(f_+(q)+f_-(q)\Big) \bigg)\, dq +  \frac{59}{25\cdot 100}\,a\,,
\]
thus finally getting~(\ref{claim}).

\step{V}{The lower bound: general case.}
In this last step we conclude the proof of the Theorem. Thanks to the above steps, we already know that the result holds in the case of four rounded corners, hence we can now assume that $\partial\C$ has only two or three arcs. In this case, there exist two maximal numbers $a^\pm\leq a$ such that $\C \subseteq \L\big(\Gamma\times (-a^-,a^+)\big)$. Let us now introduce a new strip $\Omega_{\widetilde\Gamma,\tilde a}$ by
\begin{align*}
\bar t:=\frac{a^+-a^-}{2}\,,&& \widetilde \Gamma:=\L\Big(\Gamma\times \{\bar t\}\Big)\,, && \tilde a := \frac{a^++a^-}{2}\,.
\end{align*}
Notice that there is a bijective map $\varphi:\Gamma\to\widetilde\Gamma$ given by $\varphi(q) = \L\big(q,\bar t\big)$, and that since $\widetilde\Gamma$ is by construction parallel to $\Gamma$, then the normal vector $N(q)$ to $\Gamma$ at $q$ coincides with the normal vector $\widetilde N\big(\varphi(q)\big)$ to $\widetilde \Gamma$ at $\varphi(q)$. Thus, being $\Omega_{\widetilde\Gamma,\tilde a}$ a subset of $\Omega_{\Gamma,a}$, the injectivity condition~(\ref{Ass.basic}) trivially holds also for $\widetilde\Gamma$ and $\tilde a$, and we can conclude that the strip $\Omega_{\widetilde\Gamma,\tilde a}$ is admissible for our purposes.\par
By construction, we have $\C\subseteq \Omega_{\widetilde \Gamma,\tilde a}$, hence $\C$ is also the Cheeger set of $\Omega_{\widetilde\Gamma,\tilde a}$. Moreover, by maximality of $a^\pm$ we know that $\C$ touches all the four parts of the boundary of $\Omega_{\widetilde\Gamma,\tilde a}$, so the preceding steps, and in particular~(\ref{cp}), allow to deduce that
\[
h\big(\Omega_{\Gamma,a}\big) = \frac{P(\C)}{|\C|} = h\big(\Omega_{\widetilde\Gamma,\tilde a}\big) \geq \frac{1}{\tilde a} + \frac{1}{200\big|\widetilde \Gamma\big|}\,.
\]
Finally, by definition $\tilde a\leq a$, while
\[
\big| \widetilde\Gamma\big| = \int_\Gamma \Big(1 -\bar t\kappa(q)\Big)\,dq \leq 2 \big|\Gamma\big|\,.
\]
Thus, we get~(\ref{strip.bounded}) with the constant $c=1/400$.
\end{proof}

\section{Solvable models}

In this section we discuss our results on the basis of several examples of curved strips about circles and circular arcs. They are referred to as \emph{solvable models} since the determination of the Cheeger constant and the Cheeger set is reduced to solving an explicit algebraic equation. Where the exact solution is not available, we have solved the problem with help of standard numerical tools.

\bigskip
\subsection{Annuli}
Probably the simplest example is given by annuli,
\emph{i.e.}\ strips built about (full) circles,
see Figure~\ref{Fig.annulus}.
Then the Cheeger set is the strip itself
and the Cheeger constant equals the half of the distance
between the boundary curves.
It follows from out Theorem~\ref{Thm.Cheeger} that exactly
the same situation holds for general curved annuli. Let us remark that also
discs can be thought as examples of curved strips.
Indeed, a disc with its central point removed
has the same Cheeger set (up to the point)
and Cheeger constant as the disc,
and the former set can be considered
as the limit case of the annulus built about the circle
of radius $a+\eps$ when $\eps \to 0+$.

\begin{figure}[htbp]
\begin{center}
\includegraphics[width=0.2\textwidth]{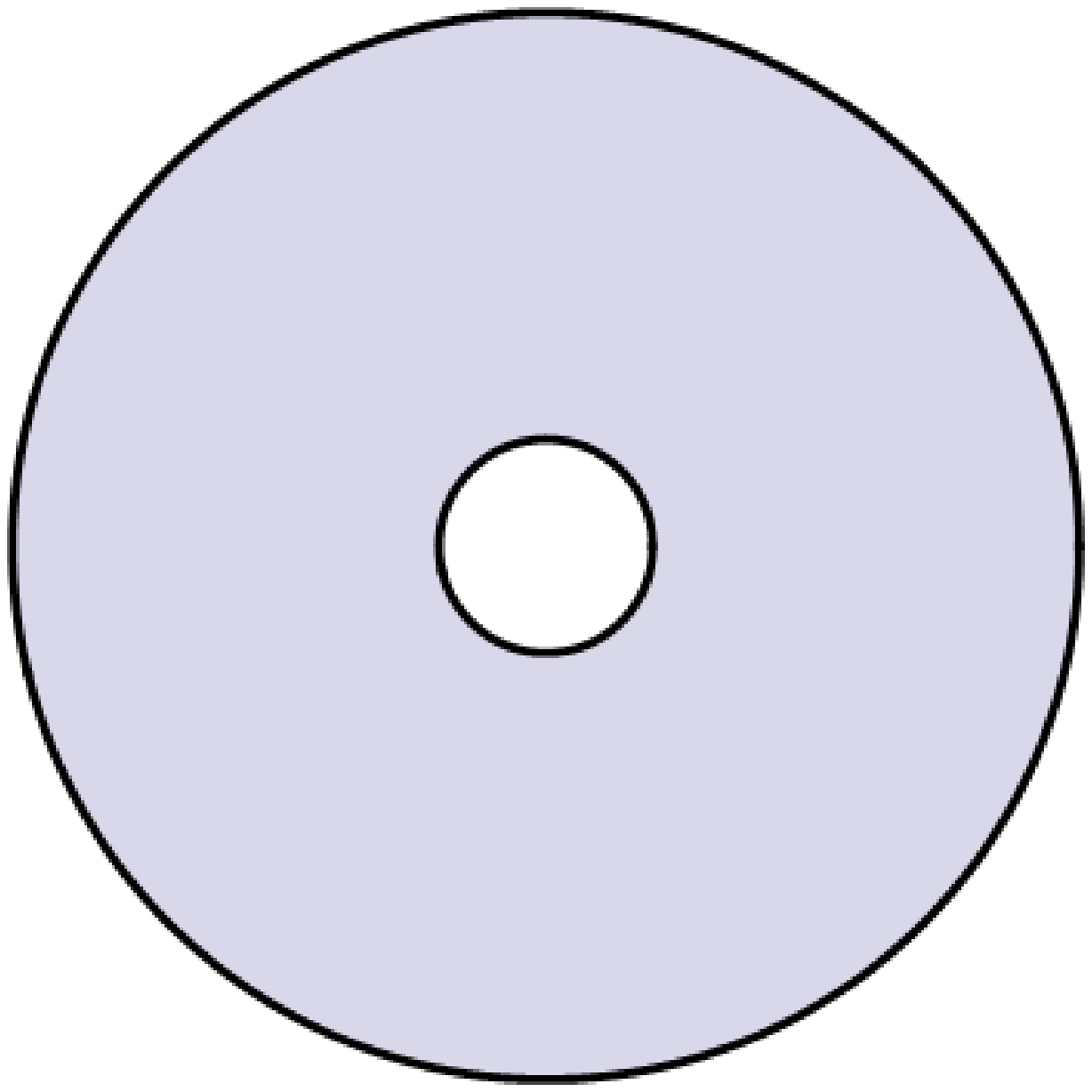}
\qquad
\includegraphics[width=0.2\textwidth]{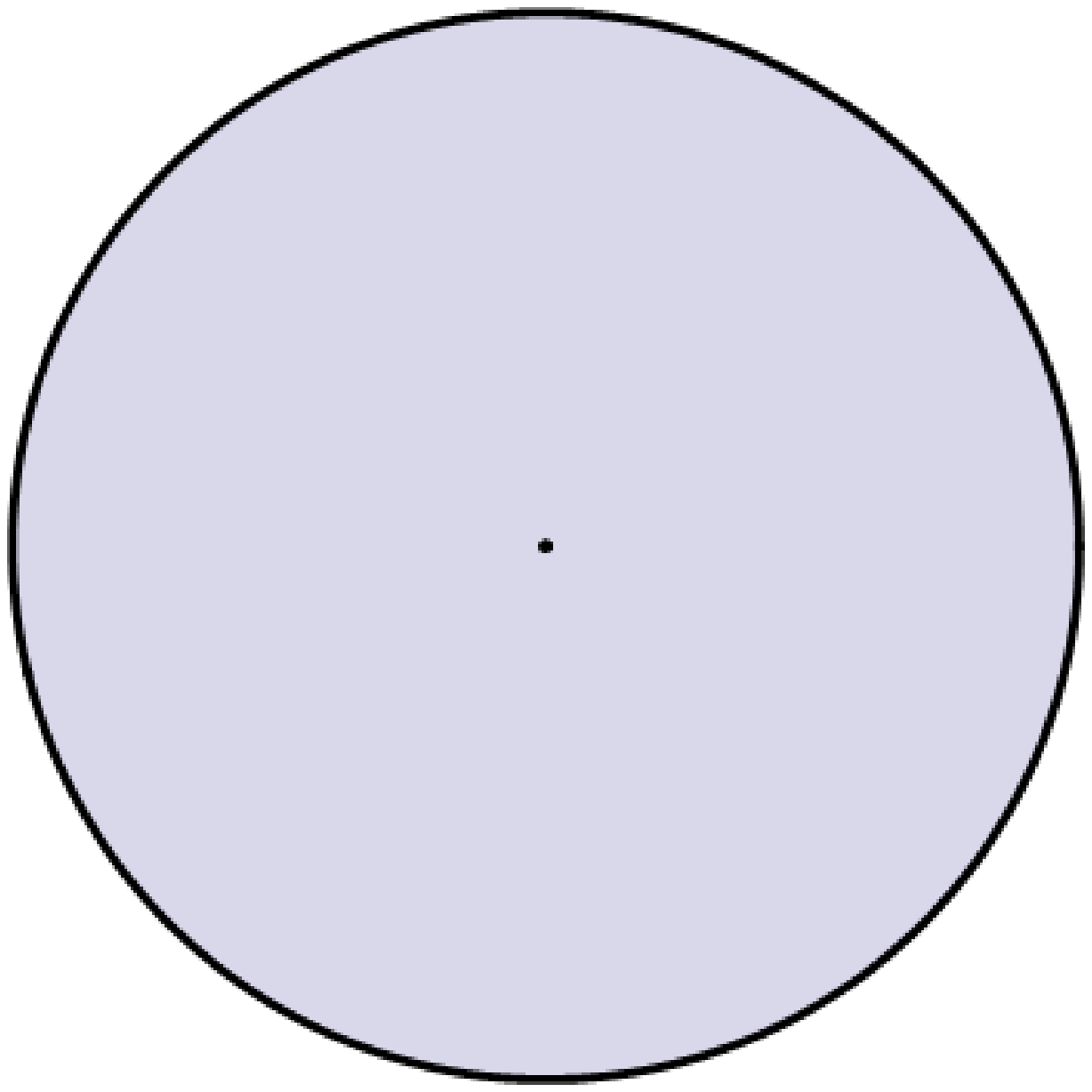}
\end{center}
\caption{The annulus and the disc considered as its limit case}
\label{Fig.annulus}
\end{figure}

\subsection{Rectangles}
The rectangle $\mathcal{R}_{a,b}:=(-b,b)\times(-a,a)$, with $a,b > 0$,
can be considered as a strip built about the segment $\Gamma:=(-b,b)\times\{0\}$.
Using Theorem~\ref{Thm.convex}, it is easy to find its Cheeger constant
explicitly:
$$
  h(\mathcal{R}_{a,b}) =
  \frac{a + b + \sqrt{(a - b)^2 + \pi a b}}{2 a b}
  \,.
$$
Notice the scaling $h(\mathcal{R}_{a,b})=h(\mathcal{R}_{1,b/a})$.
The procedure also determines the Cheeger set of~$\mathcal{R}_{a,b}$
as the rectangle with its corners rounded off by circular arcs
of radius $h(\mathcal{R})^{-1}$, see Figure~\ref{Fig.rectangle}.

\begin{figure}[htbp]
\begin{center}
\includegraphics[width=0.4\textwidth]{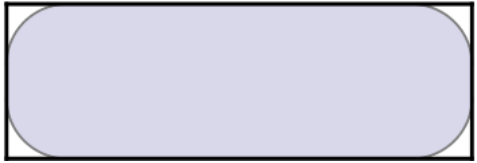}
\end{center}
\caption{The rectangle and its Cheeger set (light gray) for $b/a=3$}
\label{Fig.rectangle}
\end{figure}

The Cheeger constant can be written as
$$
  h(\mathcal{R}_{a,b}) = \frac{1}{a} + \frac{k(a,b)}{|\Gamma|}
  \,, \qquad \mbox{where} \qquad
  k(a,b) :=
  \frac{a - b + \sqrt{(a - b)^2 + \pi a b}}{a}
$$
and $|\Gamma|=2b$. Notice the scaling $k(a,b)=k(1,b/a)$.
It is straightforward to check that $b/a \mapsto k(a,b)$
is a decreasing function with the limits
$k(a,b) \to 2$ as $b/a \to 0$
and $k(a,b) \to \pi/2$ as $b/a \to \infty$.
Hence the upper bound of Theorem~\ref{Thm.bounded}
becomes sharp in the limit of very narrow rectangles.
The dependence of the Cheeger constant~$h$ and of the quantity~$k$
on rectangle parameters is shown in Figure~\ref{Fig.rectangle.graphs}.

\begin{figure}[htbp]
\begin{center}
\includegraphics[width=0.4\textwidth]{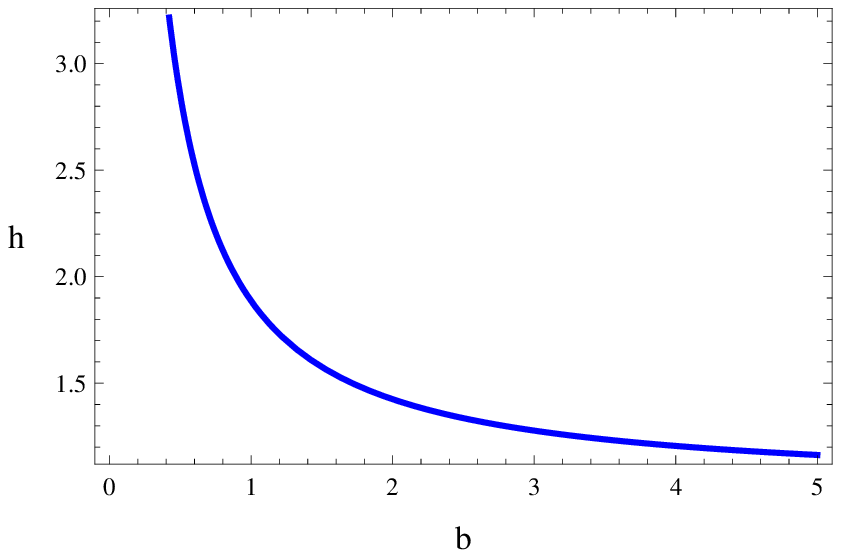}
\qquad
\includegraphics[width=0.4\textwidth]{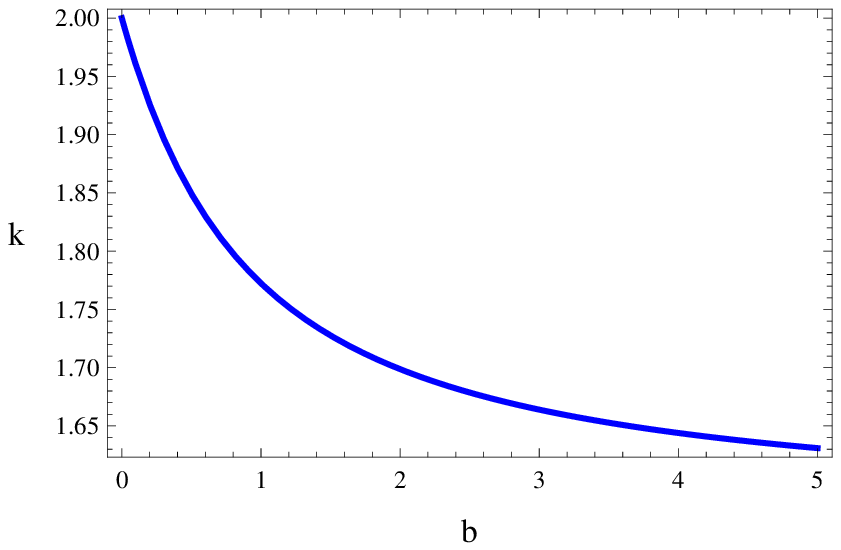}
\end{center}
\caption{The Cheeger constant~$h$
and the quantity~$k$ for rectangles with $a=1$
}
\label{Fig.rectangle.graphs}
\end{figure}

\subsection{Sectors}
Let~$\Gamma_a$ be the circle of curvature $\kappa=a^{-1}$
and consider its part~$\Gamma_a^{\alpha}$ of length
$|\Gamma_a^\alpha| = \alpha a$,
with any $\alpha\in(0,2\pi)$, see Figure~\ref{Fig.sector}.
The corresponding strip $\Omega_a^\alpha := \Omega_{\Gamma_a^\alpha,a}$
does not satisfy the assumption~\eqref{Ass.basic}.
However, since $\L$ is in fact injective in $\Gamma_a\times (-a,a)$,
it can be considered as a limit case
of admissible strips along corresponding parts of the circle
of radius $a+\eps$ when $\eps \to 0+$.

\begin{figure}[htbp]
\begin{center}
\includegraphics[width=0.4\textwidth]{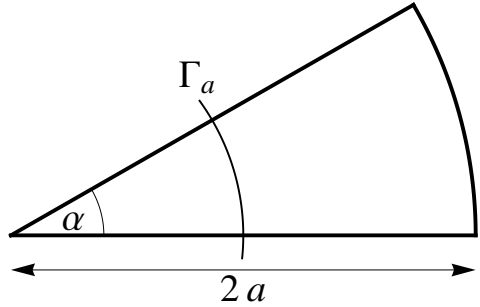}
\end{center}
\caption{The sector of a disc considered as a strip
built about the $(\frac{\alpha}{2\pi})^\mathrm{th}$-part
of a circle
}
\label{Fig.sector}
\end{figure}

The Cheeger constant and the Cheeger set of~$\Omega_a^\alpha$
can be found as follows.
Firstly, we construct a family of domains~$S_r$, with $r\in(0,a)$,
defined by rounding off the corners in~$\Omega_a^\alpha$
of angle smaller than~$\pi$ by circular arcs of radius~$r$.
This can be done by a straightforward usage of elementary geometric rules.
Secondly, we minimize the quotient $P(S_r)/|S_r|$ with respect to~$r$,
which is done with help of a numerical optimization.
The minimum of the quotient corresponds to the Cheeger constant
and the minimizer is the Cheeger set.
The procedure is equivalent to using Theorem~\ref{Thm.convex},
which seems to remain valid also for $\alpha > \pi$,
corresponding to non-convex sectors.

In view of the obvious scaling
$h(\Omega_a^\alpha) = h(\Omega_1^\alpha) / a$,
one can restrict to $a=1$, without loss of generality.
The dependence of the Cheeger constant on~$\alpha$
is shown in Figure~\ref{Fig.sector.graphs}.
Table~\ref{Tab.sector} contains numerical values
for some specific angles.

Writing the Cheeger constant as
$$
  h(\Omega_a^\alpha) = \frac{1}{a} + \frac{k(\alpha)}{|\Gamma_a^\alpha|}
  \,,
$$
we also study the dependence of the constant~$k(\alpha)$ on~$\alpha$,
see Figure~\ref{Fig.sector.graphs} and Table~\ref{Tab.sector}.
The third value of~$\alpha$ in Table~\ref{Tab.sector} corresponds
to the maximal point of the curve $\alpha \mapsto k(\alpha)$
from Figure~\ref{Fig.sector.graphs}.
In any case, we see that the upper bound of Theorem~\ref{Thm.bounded}
is quite good for all the sectors. Finally, Figure~\ref{Fig.sector.sets} shows a numerical approximation of the Cheeger sets for some annuli.

\begin{figure}[htbp]

\begin{center}
\includegraphics[width=0.4\textwidth]{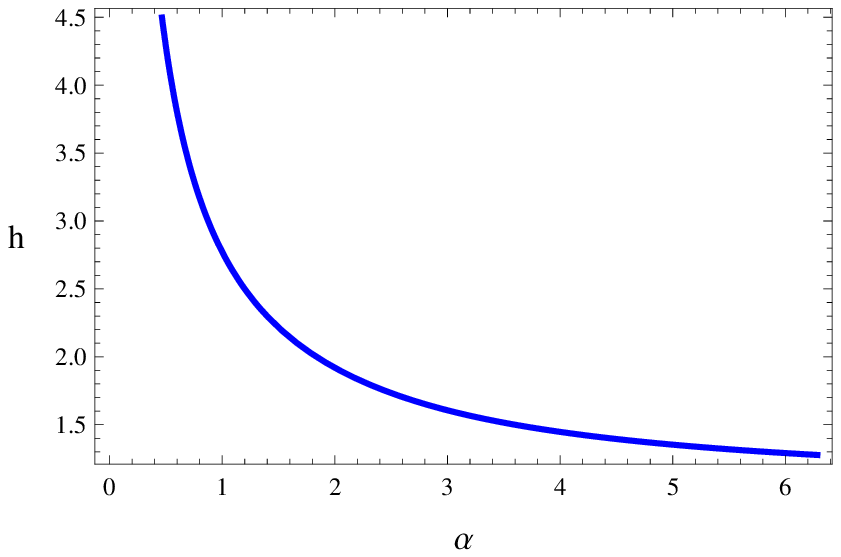}
\qquad
\includegraphics[width=0.4\textwidth]{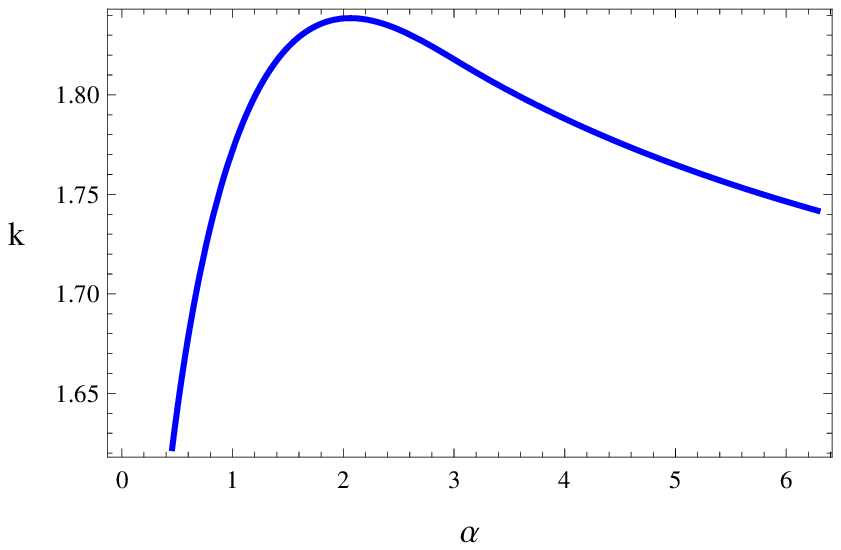}
\end{center}
\caption{The Cheeger constant~$h$ and the constant~$k$ for sectors with $a=1$}
\label{Fig.sector.graphs}
\end{figure}
\begin{table}[htbp]
\begin{center}
\begin{tabular}{|c||c|c|c|c|c|c|c|}
\hline
$\alpha$
& $\pi/10$
& $\pi/2$
& $0.656749 \, \pi$
& $3\pi/4$
& $\pi$
& $3\pi/2$
& $2\pi$
\\ \hline
$h(\Omega_1^\alpha)$
& 5.92687
& 2.16358
& 1.89111
& 1.77915
& 1.57714
& 1.37582
& 1.27722
\\ \hline
$k(\Omega_1^\alpha)$
& 1.54782
& 1.82774
& 1.83856
& 1.83583
& 1.81315
& 1.77101
& 1.74184
\\ \hline
\end{tabular}
\end{center}
\caption{The Cheeger constant~$h$
and the constant~$k$ for sectors with $a=1$
}
\label{Tab.sector}
\end{table}
\begin{figure}[htbp]
\begin{center}
\begin{tabular}{ccc}
\includegraphics[width=0.3\textwidth]{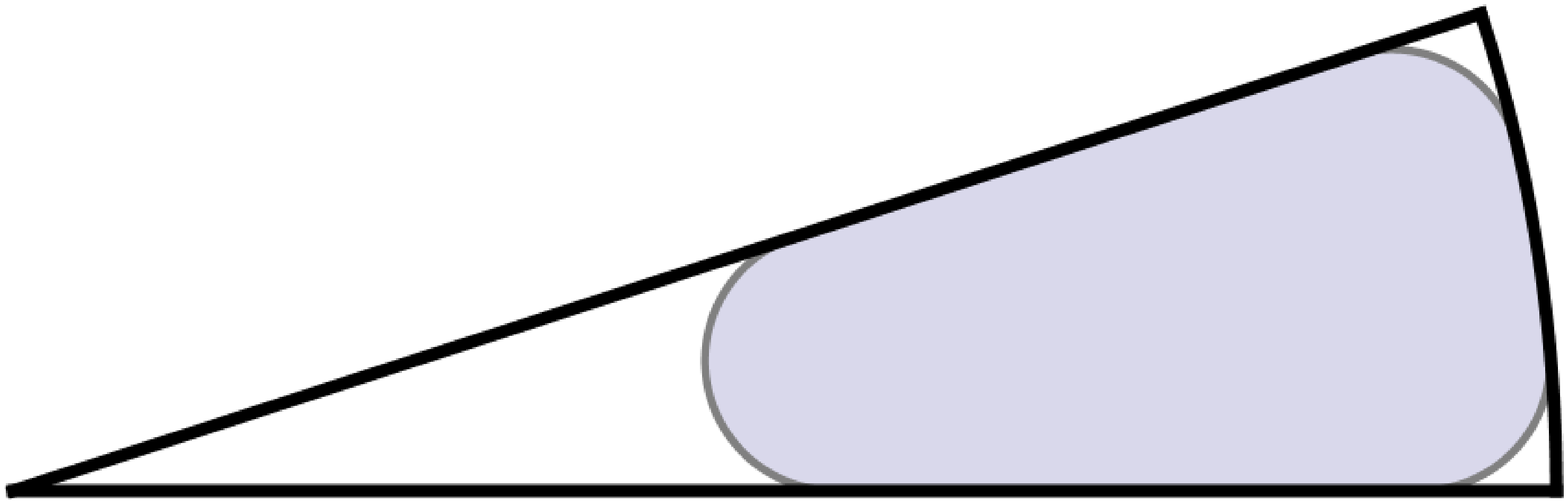}
& \includegraphics[width=0.2\textwidth]{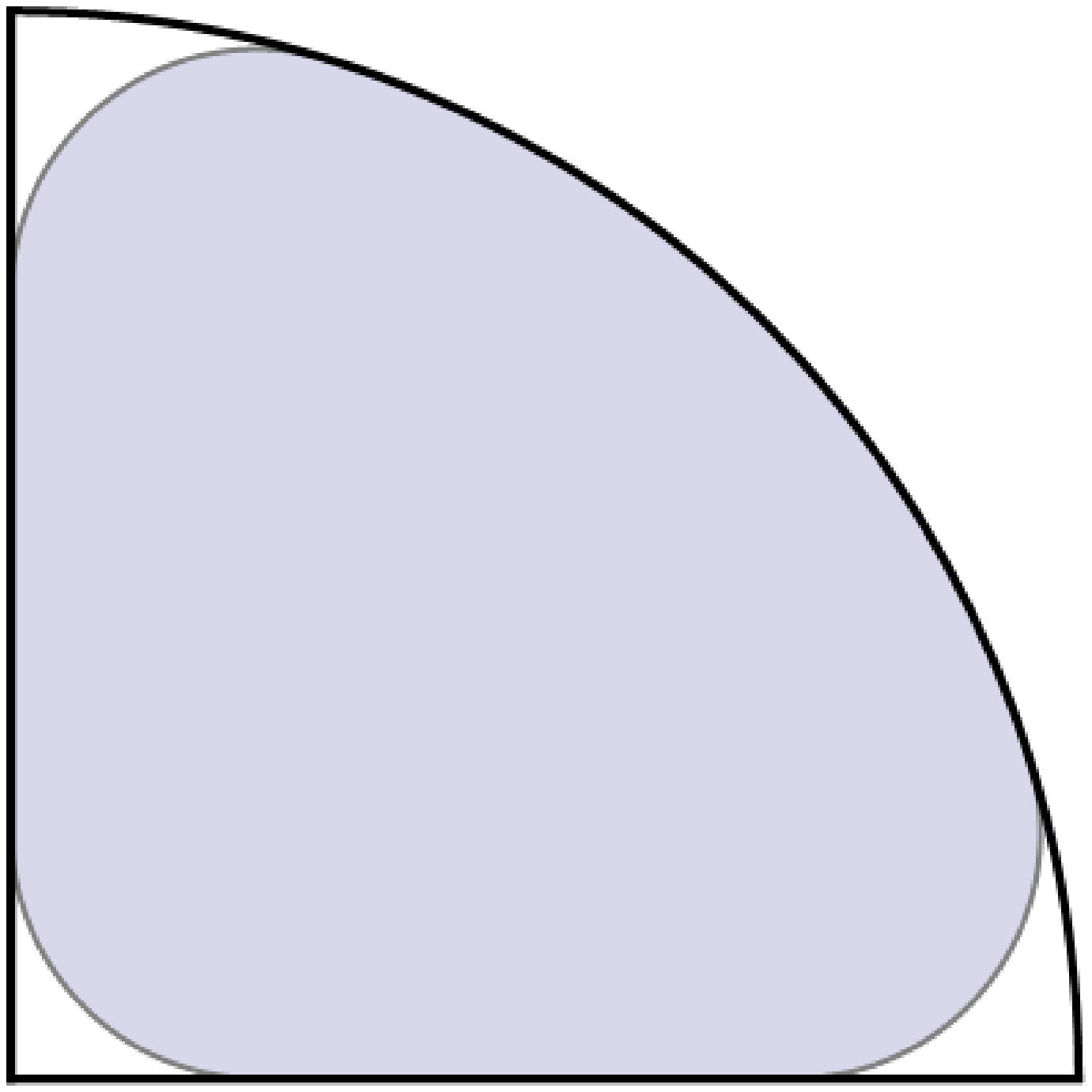}
& \includegraphics[width=0.3\textwidth]{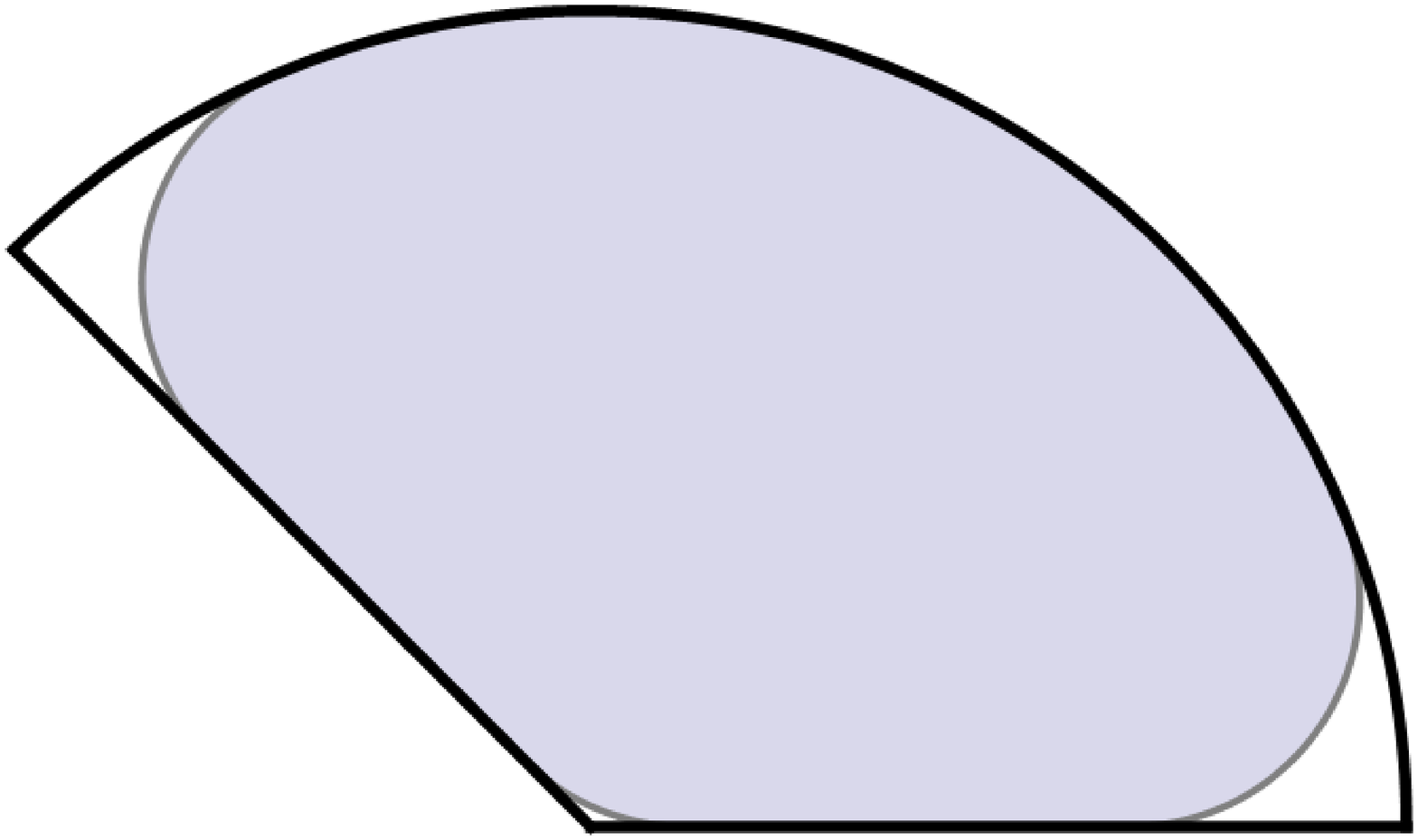}
\\
$\alpha=\pi/10$
& $\alpha=\pi/2$
& $\alpha=3\pi/4$
\bigskip \\
\includegraphics[width=0.3\textwidth]{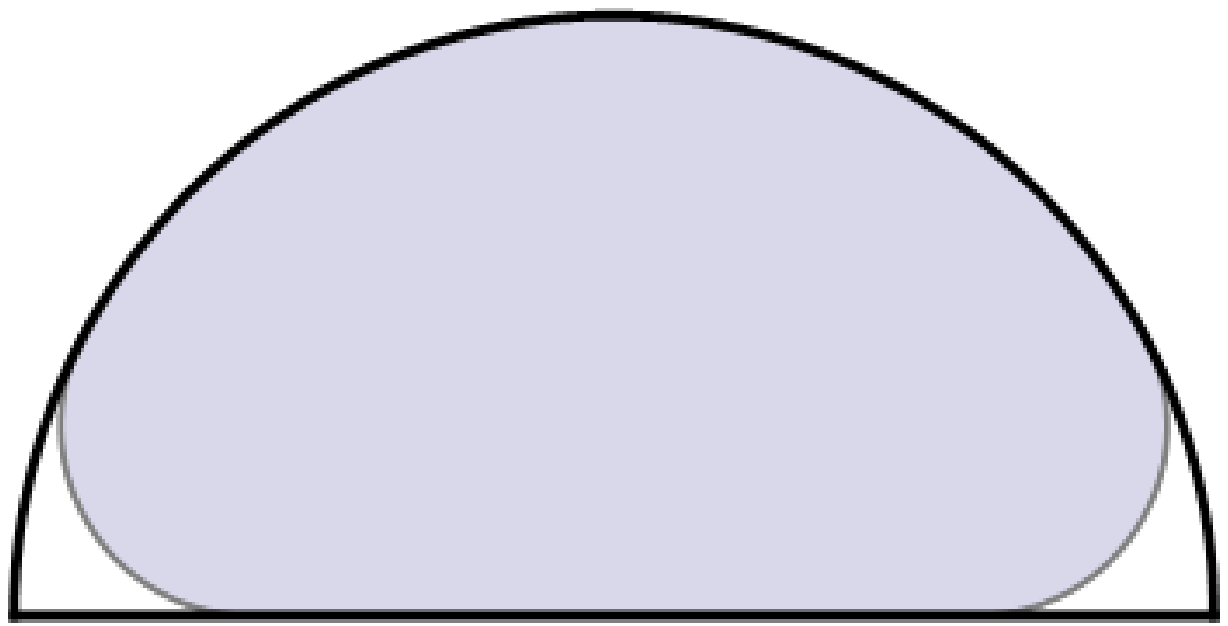}
& \includegraphics[width=0.3\textwidth]{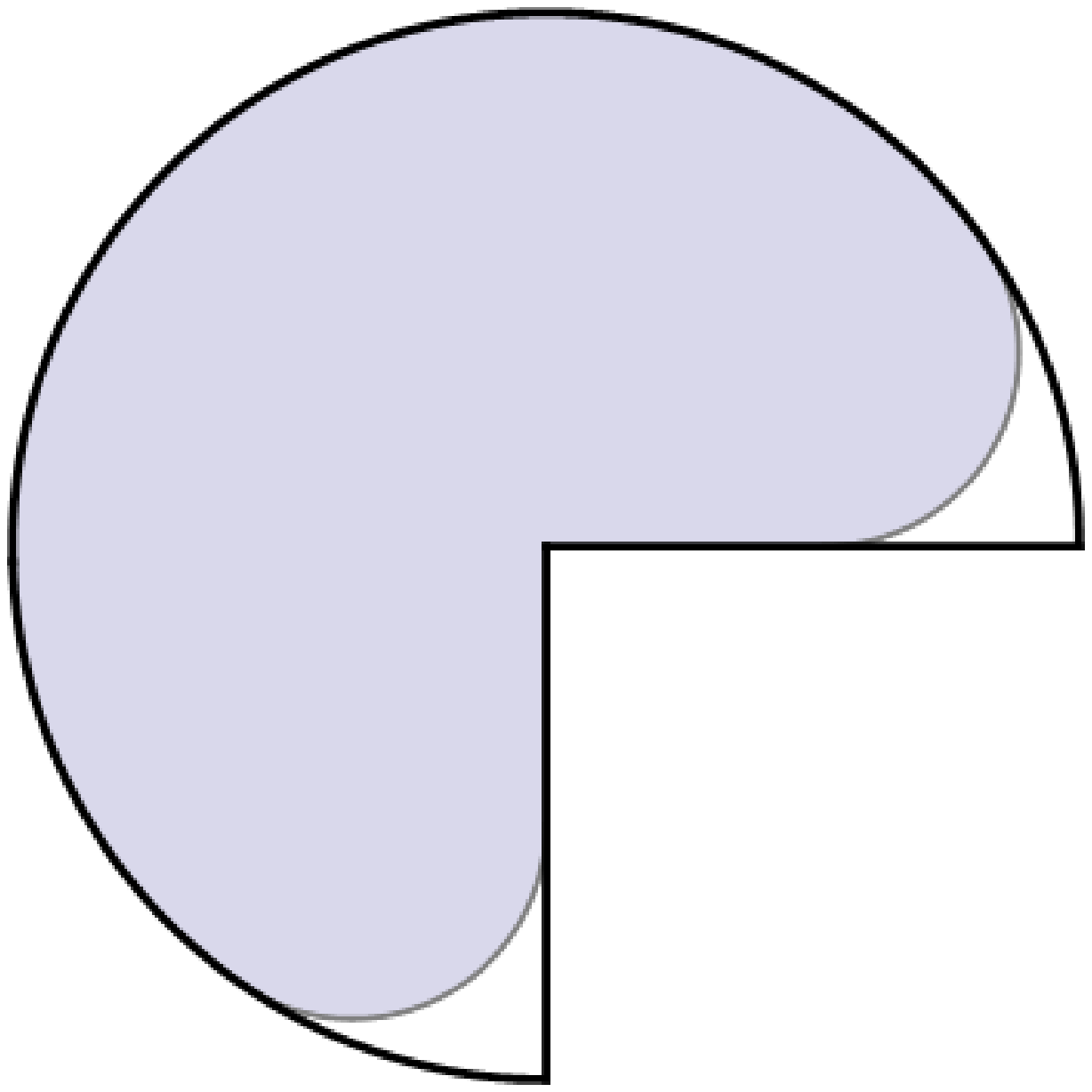}
& \includegraphics[width=0.3\textwidth]{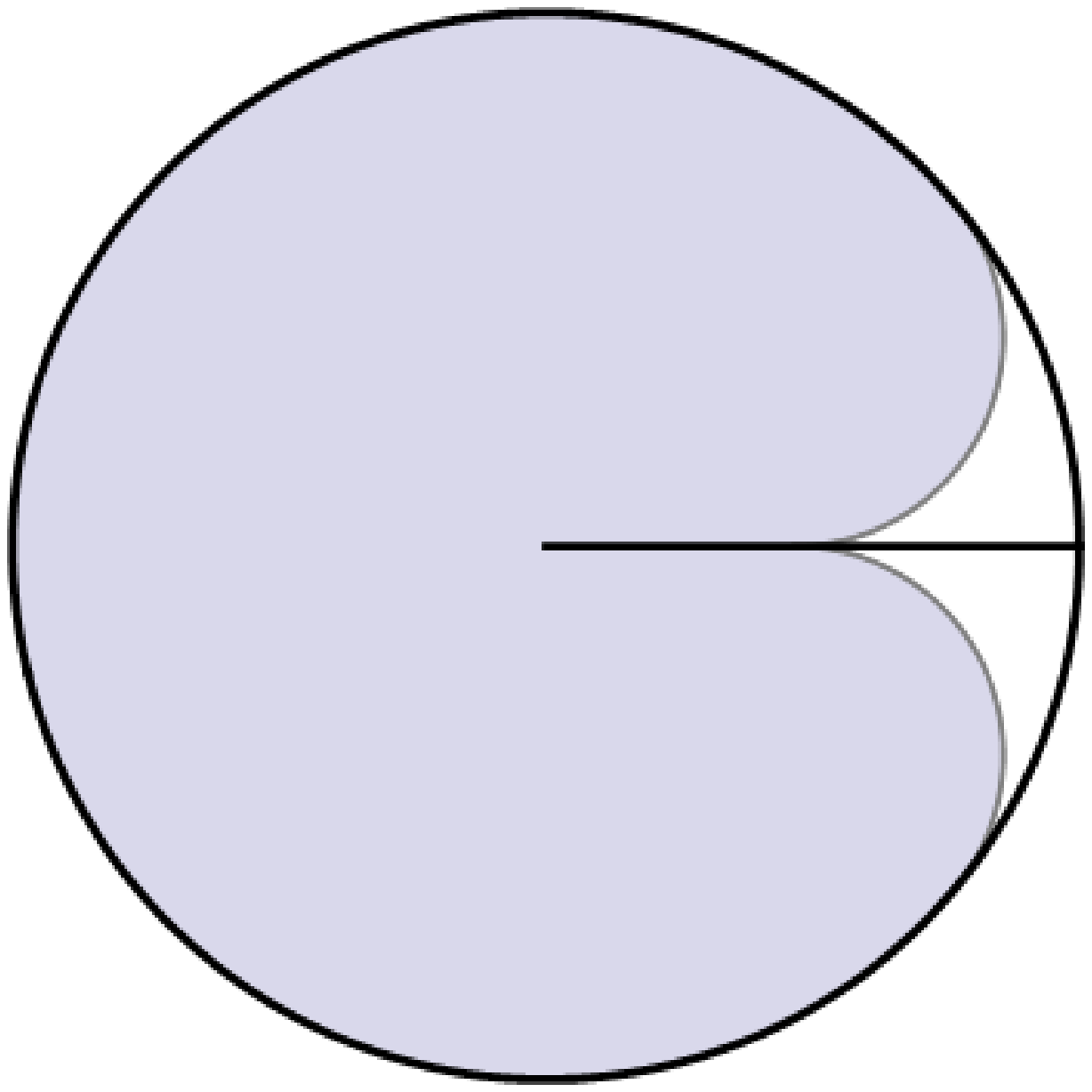}
\\
$\alpha=\pi$
& $\alpha=3\pi/2$
& $\alpha=2\pi$
\end{tabular}
\end{center}
\caption{The sectors and their Cheeger sets (light gray)
}
\label{Fig.sector.sets}
\end{figure}

\subsection*{Acknowledgement}
The authors acknowledge the hospitality of the Nuclear Physics Institute ASCR
in \v{R}e\v{z} and the University of Pavia.
This work was partially supported by the Czech Ministry of Education, Youth and Sports within the project LC06002, by the Spanish Ministry of Education and Science within the project MTM2008-03541, and by the ERC within the Advanced Grant n. 226234 and the Starting Grant n. 258685.


\begin{thebibliography}{100}
\bibitem{Alter-Caselles-Chambolle_2005a} F.~Alter, V.~Caselles, and A.~Chambolle, \emph{Evolution of characteristic functions of convex sets in the plane by the minimizing total variation flow}, Interfaces Free Bound. \textbf{7} (2005), 29--53.
\bibitem{AFP} L. Ambrosio, N. Fusco, and D. Pallara, Functions of Bounded Variation and Free Discontinuity Problems, Oxford University Press (2000).
\bibitem{Bellettini-Caselles-Novaga_2002} G.~Bellettini, V.~Caselles, and M.~Novaga, \emph{The total variation flow in {$\mathbb{R}^N$}}, J. Differential Equations \textbf{184} (2002), 475--525.
\bibitem{CCN} V. Caselles, A. Chambolle, and M. Novaga, Uniqueness of the Cheeger set of a convex body, Pacific J. Math. \textbf{232} (2007), 77--90.
\bibitem{DE} P.~Duclos and P.~Exner, \emph{{C}urvature-induced bound states in quantum waveguides in two and three dimensions}, Rev. Math. Phys. \textbf{7} (1995), 73--102.
\bibitem{Evans-Gariepy} L.~C. Evans and R.~F. Gariepy, \emph{Measure theory and fine properties of functions}, CRC Press, Boca Raton, 1992.
\bibitem{EFK} P.~Exner, P.~Freitas, and D.~Krej\v{c}i\v{r}\'{\i}k, \emph{A lower bound to the spectral threshold in curved tubes}, R.~Soc. Lond. Proc. Ser.~A Math. Phys. Eng. Sci. \textbf{460} (2004), no.~2052, 3457--3467.
\bibitem{Grieser_2006} D.~Grieser, \emph{The first eigenvalue of the {L}aplacian, isoperimetric constants, and the {M}ax {F}low {M}in {C}ut {T}heorem}, Arch. Math. \textbf{87} (2006), 75--85.
\bibitem{Hebey-Saintier_2006} E.~Hebey and N.~Saintier, \emph{Stability and perturbation of the domain for the first eigenvalue of the {$1$}-{L}aplacian}, Archiv der Math. \textbf{86} (2006), 340--351.
\bibitem{Kawohl-Fridman_2003} B.~Kawohl and V.~Fridman, \emph{Isoperimetric estimates for the first eigenvalue of the {$p$}-{L}aplace operator and the {C}heeger constant}, Comment. Math. Univ. Carolinae \textbf{44} (2003), 659--667.
\bibitem{Kawohl-Lachand-Robert_2006} B.~Kawohl and T.~Lachand-Robert, \emph{Characterization of {C}heeger sets for convex subsets of the plane}, Pacific J. Math. \textbf{225} (2006), 103--118.
\bibitem{KK} D.~Krej\v{c}i\v{r}\'{\i}k and J.~K\v{r}\'{\i}\v{z}, \emph{On the spectrum of curved quantum waveguides}, Publ.~RIMS, Kyoto University \textbf{41} (2005), no.~3, 757--791.
\bibitem{Saintier} N.~Saintier, \emph{Shape derivative of the first eigenvalue of the {$1$}-{L}aplacian}, arXiv:0706.0873v1 [math.AP] (2007).
\end{thebibliography}
\end{document}